\documentclass[11pt]{amsart}

\usepackage{amsfonts,latexsym}
\usepackage{amsmath}
\usepackage{amstext}
\usepackage{amssymb}
\usepackage{graphicx,diagrams}
\graphicspath{{figures/}}

\newcommand{\R}{\mathbb{R}}

\newcommand{\Z}{\mathbb{Z}}
\newcommand{\N}{{\mathbb{N}}}

\newtheorem{theorem}{Theorem}[section]

\theoremstyle{definition}
\newtheorem{definition}[theorem]{Definition}

\theoremstyle{remark}
\newtheorem{remark}[theorem]{Remark}

\numberwithin{equation}{section}

\begin{document}

\title[Approximately Bisimilar Symbolic Models for Incrementally Stable Switched Systems]
      {Approximately Bisimilar Symbolic Models\\ for Incrementally Stable Switched Systems}
\thanks{This work was partially supported by the ANR SETIN project VAL-AMS 
        and by the NSF CAREER award 0717188.}


\author[Antoine Girard]{Antoine Girard}
\address{Laboratoire Jean Kuntzmann \\
Universit\'e de Grenoble \\
B.P. 53, 38041 Grenoble, France} \email{antoine.girard@imag.fr}

\author[Giordano Pola]{Giordano Pola}
\address{Department of Electrical and Information Engineering \\
University of L'Aquila \\
Poggio di Roio, 67040 L'Aquila, Italy} \email{pola@ing.univaq.it}

\author[Paulo Tabuada]{Paulo Tabuada}
\address{Department of Electrical Engineering \\
University of California at Los Angeles \\
Los Angeles, CA 90095-1594 } \email{tabuada@ee.ucla.edu}


\maketitle


\begin{abstract}

Switched systems constitute an important modeling paradigm faithfully describing many engineering systems in which software interacts with the physical world. Despite considerable progress on stability and stabilization of switched systems, the constant evolution of technology demands that we make similar progress with respect to different, and perhaps more complex, objectives. This paper describes one particular approach to address these different objectives based on the construction of approximately equivalent (bisimilar) symbolic models for  switched systems.   
The main contribution of this paper consists in showing that under standard assumptions ensuring incremental stability of a switched system (i.e. existence of a common Lyapunov function, or multiple Lyapunov functions with dwell time), it is 
possible to construct a finite symbolic model that is approximately bisimilar to the original switched system with a precision
that can be chosen a priori.  To support the computational merits of the proposed approach, we use symbolic models to synthesize controllers for two examples of switched systems,
including the boost DC-DC converter.
\end{abstract}
\section{Introduction}
Switched systems constitute an important modeling paradigm faithfully describing many engineering systems in which software interacts with the physical world. Although this fact already amply justifies its study, switched systems are also quite intriguing from a theoretical point of view. It is well known that by judiciously switching between stable subsystems one can render the overall system unstable. This motivated several researchers over the years to understand which classes of switching strategies or switching signals preserve stability (see e.g.~\cite{liberzon2003}). 
Despite considerable progress on stability and stabilization of switched systems, the constant evolution of technology demands that we make similar progress with respect to different, and perhaps more complex, objectives. These comprise the synthesis of control strategies guiding the switched systems through predetermined operating points while avoiding certain regions in the state space, enforcing limit cycles and oscillatory behavior, reconfiguration upon the occurrence of faults, etc. 

This paper describes one particular approach to address these different objectives based on the construction of {\it symbolic models} that are abstract description of the switched dynamics and in which each  abstract state, or symbol, corresponds to an aggregate of 
states in the switched system. When the symbolic models are finite, controller synthesis problems can be efficiently solved by resorting to mature techniques developed in the areas of supervisory control of discrete-event systems~\cite{ramadge1987} and algorithmic game theory~\cite{arnold2003}. The crucial step is therefore the construction of symbolic models that are detailed enough to capture all the behavior of the original system, but not so detailed that their use for synthesis is as difficult as the original model. This is accomplished, at the technical level, by using the notion of approximate bisimulation.
Approximate bisimulation has been introduced in~\cite{girard2007}, as an approximate version of the usual bisimulation relation~\cite{milner1989,park1981}, and in~\cite{SymbolicControl06} by using set-valued observations. It generalizes the notion of bisimulation by requiring the outputs of two systems to be close instead of being strictly equal. This relaxed requirement makes it possible to compute symbolic models for larger classes of systems as shown recently for incrementally stable continuous control systems~\cite{pola2007}.

In this paper, we first extend the standard theorems on asymptotic stability of switched systems, 
i.e. results based on the existence a common Lyapunov function, or multiple Lyapunov functions with dwell time~\cite{liberzon2003}, to study incremental stability of switched systems.
The main contribution of the paper consists in showing that under the assumptions ensuring incremental stability of a switched system, it is 
possible to construct a symbolic model that is approximately bisimilar to the original switched system with a precision
that can be chosen a priori. The proof is constructive and it is straightforward to derive a  procedure for the computation of these symbolic models.
Since in problems of practical interest the state space can be assumed to be bounded, the resulting symbolic model is guaranteed to have finitely many states and can thus be used for algorithmic controller synthesis. 
To support the computational merits of the proposed approach, we show how to use symbolic models to synthesize controllers for two examples of switched systems. First, we consider 
the boost DC-DC converter, and show how to synthesize a switched controller that regulates the output voltage at a desired level. For this example, it is possible to find a common Lyapunov function, therefore, 
we consider a second example that illustrates the use of multiple Lyapunov functions with dwell time. 
A preliminary version of these results appeared in~\cite{girard2008}.

In the following, the symbols $\N$, $\Z$, $\R$, $\R^+$ and $\R^+_0$ denote the set of
natural, integer, real, positive and nonnegative real numbers
respectively. 
Given a vector $x\in \R^n$, we denote by $x_i$ its $i$-th
coordinate and by $\|x\|$ its Euclidean norm.

\section{Switched systems and incremental stability}

\subsection{Switched systems}

We shall consider the class of switched systems
formalized in the following definition.

\begin{definition}
\label{def:switch} A switched system is a quadruple
$
\Sigma=(\R^n,P,\mathcal P,F),
$
where:
\begin{itemize}
\item $\R^n$ is the state space;

\item $P =\{1,\dots,m\}$ is the finite set of modes;

\item $\mathcal P$ is a subset of $\mathcal{S}(\R^+_0,P)$ which denotes the set of piecewise constant functions from $\R^+_0$ to $P$, continuous from the right and with a finite number of discontinuities on every bounded interval of $\R^+_0$;

\item $F=\{f_{1},\dots,f_{m}\}$ is a collection of vector
fields indexed by $P$. For all $p \in P$, $f_p :
\R^n \rightarrow \R^n$ is a locally Lipschitz continuous map.
\end{itemize}
\end{definition}

For all $p \in P$, we denote by $\Sigma_p$ the continuous {\it subsystem} of $\Sigma$ defined by the differential equation:
\begin{equation}
\label{eq:edo}
\dot{\mathbf{x}}(t)=f_{p}(\mathbf{x}(t)).
\end{equation}
We make the assumption that the vector field $f_p$ is such that the solutions of the differential equation (\ref{eq:edo}) are defined on an interval of the form $]a,+\infty[$ with $a<0$. 
Necessary and sufficient conditions to be satisfied by $f_p$ can be found in~\cite{angeli1999}.
Simpler, but only sufficient, conditions include linear growth or compact support of the vector field $f_p$.

A {\it switching signal} of $\Sigma$ is a function $\mathbf{p}\in \mathcal P$, the discontinuities of $\mathbf{p}$
are called {\it switching times}.
A piecewise $\mathcal C^1$ function $\mathbf{x}:\R_0^+ \rightarrow \R^n$ is
said to be a {\it trajectory} of $\Sigma$ if it is continuous and there exists a switching
signal $\mathbf p\in \mathcal P$ such that, at each
$t\in \R_0^+$ where the function $\mathbf{p}$ is continuous, $\mathbf{x}$ is continuously differentiable and satisfies:
$$
\dot{\mathbf{x}} (t) = f_{\mathbf{p}(t)}(\mathbf{x}(t)).
$$
We will use $\mathbf{x}(t,x,\mathbf{p})$ to denote the point
reached at time $t\in \R_0^+$ from the initial condition $x$ under
the switching signal $\mathbf{p}$. 
The assumptions on the vector fields $f_{1},\dots,f_{m}$ ensure for all initial conditions and switching
signals, existence and uniqueness of the trajectory of $\Sigma$. Furthermore since switching signals have only a finite number of discontinuities on every bounded interval, Zeno behaviors are ruled out.
Let us remark that a trajectory of $\Sigma_p$ is a trajectory of $\Sigma$  associated with the constant switching signal $\mathbf{p}(t)=p$, for all $t\in \R_0^+$. Then, we will use $\mathbf{x}(t,x,{p})$ to denote the point reached by $\Sigma_p$ at time $t\in \R^+_0$ from the initial condition $x$.

\subsection{Incremental stability}

The results presented in this paper rely on some stability
notions. 
A continuous function $\gamma :\R_0^+ \rightarrow \R_0^+ $ is said
to belong to class $\mathcal K$ if it is strictly increasing and
$\gamma(0)=0$. Function $\gamma$ is said to belong to class $\mathcal
K_\infty$ if it is a $\mathcal K$ function and
$\gamma(r)\rightarrow \infty$ when $r \rightarrow \infty$. A
continuous function $\beta:\R^+_0 \times \R^+_0 \rightarrow
\R^+_0$ is said to belong to class $\mathcal{KL}$ if for all fixed
$s$, the map $r\mapsto \beta(r,s)$ belongs to class $\mathcal
K_\infty$ and for all fixed $r$, the map $s \mapsto \beta(r,s)$ is
strictly decreasing and $\beta(r,s)\rightarrow 0$ when
$s\rightarrow \infty$.

\begin{definition}~\cite{angeli2002} The subsystem $\Sigma_p$ is
incrementally globally asymptotically stable ($\delta$-GAS) if
there exists a $\mathcal{KL}$ function $\beta_p$ such that for all
$t\in \R_0^+$, for all $x,y \in \R^n$, the following condition is
satisfied:
$$
 \|\mathbf{x}(t,x,p)-\mathbf{x}(t,y,p)\| \le
\beta_p(\|x-y\|,t).
$$
\end{definition}

Intuitively, incremental stability means that all the trajectories
of the subsystem $\Sigma_p$ converge to the same reference
trajectory independently of their initial condition. This is an
incremental version of the notion of global asymptotic
stability (GAS)~\cite{khalil1996}. Let us remark that when $f_p$
satisfies $f_p(0)=0$ then $\delta$-GAS implies GAS, as all the
trajectories of $\Sigma_p$ converge to the trajectory
$\mathbf{x}(t,0,p)=0$. Further, if $f_p$ is linear then $\delta$-GAS and GAS
are equivalent. Similarly to GAS, $\delta$-GAS can be characterized by dissipation
inequalities.
\begin{definition} A smooth function $V_p:\R^n\times \R^n \rightarrow
\R^+_0$ is a $\delta$-GAS Lyapunov function
for
$\Sigma_p$ if there exist $\mathcal K_\infty$ functions
$\underline{\alpha}_{p}$, $\overline{\alpha}_{p}$ and $\kappa_p\in \R^+$ such that:
\begin{eqnarray}
\label{eq:lyap1}
\forall x,y\in \R^n,\; &
\underline{\alpha}_{p}(\|x-y\|) \le V_p(x,y) \le
\overline{\alpha}_{p}(\|x-y\|);\\
\label{eq:lyap2}
\forall x,y\in \R^n,\; &
\frac{\partial V_p}{\partial x}(x,y) f_p(x)+ \frac{\partial
V_p}{\partial y}(x,y) f_p(y) \le -\kappa_p V_p(x,y)
.
\end{eqnarray}
\end{definition}

The following result completely characterizes $\delta$-GAS in
terms of existence of a $\delta$-GAS Lyapunov function.
\begin{theorem}~\cite{angeli2002} $\Sigma_p$ is $\delta$-GAS if and only if
it admits a $\delta$-GAS Lyapunov function.
\end{theorem}

\begin{remark}
In~\cite{angeli2002}, (\ref{eq:lyap2}) is replaced by
$
\frac{\partial V_p}{\partial x}(x,y) f_p(x)+ \frac{\partial
V_p}{\partial y}(x,y) f_p(y) \le -\rho_p(\|x-y\|),
$ 
where $\rho_p$ is a positive definite function. It is known, though not trivial to show, 
that there is no loss of generality
in considering $\rho_p(\|x-y\|)=\kappa_p V_p(x,y)$, modifying the $\delta$-GAS Lyapunov function $V_p$ if necessary
(see e.g.~\cite{praly1996}).
\end{remark}

For the purpose of this paper, we extend the notion of incremental stability to switched systems as follows:
\begin{definition} A switched system $\Sigma=(\R^n,P,\mathcal P,F)$ is
incrementally globally uniformly asymptotically stable ($\delta$-GUAS) if
there exists a $\mathcal{KL}$ function $\beta$ such that for all
$t\in \R_0^+$, for all $x,y \in \R^n$, for all switching signals $\mathbf{p} \in \mathcal P$, the following condition is
satisfied:
\begin{equation}
\label{eq:stab2}
 \|\mathbf{x}(t,x,\mathbf p)-\mathbf{x}(t,y,\mathbf p)\| \le
\beta(\|x-y\|,t).
\end{equation}
\end{definition}

Let us remark that the speed of convergence specified by the function $\beta$ is independent of the switching signal
$\mathbf p$. Thus, the stability property is  uniform over the set of switching signals; hence the notion
of incremental global uniform asymptotic stability.
Incremental stability of a switched system means that all the trajectories associated with the same switching signal
converge to the same reference trajectory independently of their initial condition. This is an  incremental version
of global uniform asymptotic stability (GUAS) for switched systems~\cite{liberzon2003}.
If for all $p\in P$, $f_p(0)=0$ (i.e. all the subsystems share a common equilibrium), then $\delta$-GUAS implies GUAS as all the trajectories of $\Sigma$ converge to the constant trajectory $\mathbf{x}(t,0,\mathbf p)=0$.
Further, if for all $p\in P$, $f_p$ is linear, $\delta$-GUAS and GUAS are equivalent.

It is well known that a switched system whose subsystems are all
GAS may exhibit some unstable behaviors under fast switching
signals. The same kind of
phenomenon can be observed for switched systems with $\delta$-GAS
subsystems. Similarly, the results on common or
multiple Lyapunov functions for proving GUAS of switched systems (see e.g.~\cite{liberzon2003}) 
can be extended to prove $\delta$-GUAS. 
Let the $\mathcal K_\infty$ functions $\underline{\alpha}$, $\overline{\alpha}$ and the
real number $\kappa$ be given by
$\underline{\alpha}=\min(\underline{\alpha}_1,\dots,\underline{\alpha}_m)$, 
$\overline{\alpha}=\max(\overline{\alpha}_1,\dots,\overline{\alpha}_m)$ and $\kappa = \min(\kappa_1,\dots,\kappa_m)$.

\begin{theorem}
\label{th:stab1}
Consider a switched system $\Sigma=(\R^n,P,\mathcal P,F)$. Let us assume that there exists $V:\R^n\times \R^n \rightarrow
\R^+_0$ which is a common 
 $\delta$-GAS Lyapunov function for subsystems
$\Sigma_1,\dots,\Sigma_m$. Then, $\Sigma$ is $\delta$-GUAS.
\end{theorem}

\begin{proof} Let $x,y\in \R^n$, $\mathbf{p} \in \mathcal P$, the function 
$t\mapsto V(\mathbf{x}(t,x,\mathbf p),\mathbf{x}(t,y,\mathbf p))$ is continuous, piecewise $\mathcal C^1$
and for all $t\in R^+_0$ where $\mathbf p$ is continuous,
equation (\ref{eq:lyap2}) gives:
\begin{eqnarray*}
\dot V(\mathbf{x}(t,x,\mathbf p),\mathbf{x}(t,y,\mathbf p))
& \le& - \kappa V(\mathbf{x}(t,x,\mathbf p),\mathbf{x}(t,y,\mathbf p)).
\end{eqnarray*}
It follows, by continuity, that for all $t\in \R^+_0$,
\begin{eqnarray*}
V(\mathbf{x}(t,x,\mathbf p),\mathbf{x}(t,y,\mathbf p)) &\le& V(\mathbf{x}(0,x,\mathbf p),\mathbf{x}(0,y,\mathbf p))e^{-\kappa t}=V(x,y)e^{-\kappa t}  \\ &\le& \overline{\alpha}(\|x-y\|) e^{-\kappa t}. 
\end{eqnarray*}
Therefore, for all $t \in \R^+_0$,
$$
\|\mathbf{x}(t,x,\mathbf p)-\mathbf{x}(t,y,\mathbf p)\| \le \underline{\alpha}^{-1} (V(\mathbf{x}(t,x,\mathbf p),\mathbf{x}(t,y,\mathbf p))) \le \underline{\alpha}^{-1} (\overline{\alpha}(\|x-y\|) e^{-\kappa t}).
$$
Then, equation (\ref{eq:stab2}) holds with the function $\beta$ given by 
$\beta(r,s)=\underline{\alpha}^{-1} (\overline{\alpha}(r) e^{-\kappa s})$. It is easy to check that $\beta$ belongs
to class $\mathcal{KL}$. Therefore, $\Sigma$ is $\delta$-GUAS.
\end{proof}

When a common  $\delta$-GAS Lyapunov function fails to exist, $\delta$-GUAS of the switched system can be ensured
by using multiple $\delta$-GAS Lyapunov functions and a restrained set of switching signals.
Let $\mathcal{S}_{\tau_d}(\R^+_0,P)$ denote the set of switching signals with 
{\it dwell time} ${\tau_d}\in \R^+_0$ so that $\mathbf p \in \mathcal{S}(\R^+_0,P)$ has dwell time $\tau_d$ if the switching times $t_1,\; t_2, \dots$ satisfy $t_1 \ge \tau_d$ and $t_{i}-t_{i-1} \ge \tau_d$, for all $i\ge 2$.

\begin{theorem}
\label{th:stab2}
Let $\tau_d \in \R^+_0$ and  consider a switched system $\Sigma_{\tau_d}=(\R^n,P,\mathcal P_{\tau_d},F)$ with $\mathcal P_{\tau_d}
 \subseteq \mathcal{S}_{\tau_d}(\R^+_0,P)$. 
 Let us assume that for all $p\in P$, there exists a $\delta$-GAS Lyapunov function $V_p$ for subsystem
$\Sigma_{\tau_d,p}$ and that in addition there exists $\mu \ge 1$ such that:
\begin{equation}
\label{eq:reset}
\forall x,y\in \R^n, \; \forall p,p' \in P, \; V_p(x,y) \le \mu V_{p'}(x,y).
\end{equation}
If $\tau_d > \frac{\log\mu}{\kappa}$, then $\Sigma_{\tau_d}$ is $\delta$-GUAS.
\end{theorem}

\begin{proof} We shall prove the $\delta$-GUAS property only for switching signals with an infinite number of discontinuities but a proof for signals with a finite number of discontinuities can be written in a very similar way.
Let $x,y\in \R^n$, $\mathbf{p} \in \mathcal P_{\tau_d}$, 
let $t_0=0$ and let $p_{i+1} \in P$ denote the value of the switching signal on the open interval $(t_{i},t_{i+1})$, for 
$i \in \N$. From equation (\ref{eq:lyap2}), for all $i\in \N$ and $t\in (t_{i},t_{i+1})$
$$
\dot V_{p_{i+1}}(\mathbf{x}(t,x,\mathbf p),\mathbf{x}(t,y,\mathbf p)) \le -\kappa V_{p_{i+1}}(\mathbf{x}(t,x,\mathbf p),\mathbf{x}(t,y,\mathbf p)).
$$
Then, for all $i\in \N$ and $t\in [t_{i},t_{i+1}]$, 
\begin{equation}
\label{eq:1}
V_{p_{i+1}}(\mathbf{x}(t,x,\mathbf p),\mathbf{x}(t,y,\mathbf p)) \le V_{p_{i+1}}(\mathbf{x}(t_{i},x,\mathbf p),\mathbf{x}(t_{i},y,\mathbf p)) e^{- \kappa (t-t_{i})}.
\end{equation}
Particularly, for $t=t_{i+1}$ and from equation (\ref{eq:reset}), it follows that for all $i \in \N$,
$$
V_{p_{i+2}}(\mathbf{x}(t_{i+1},x,\mathbf p),\mathbf{x}(t_{i+1},y,\mathbf p)) \le 
\mu e^{- \kappa (t_{i+1}-t_{i})}
V_{p_{i+1}}(\mathbf{x}(t_{i},x,\mathbf p),\mathbf{x}(t_{i},y,\mathbf p)). 
$$
Using this inequality, we prove by induction that for all $i\in \N$
\begin{equation}
\label{eq:2}
V_{p_{i+1}}(\mathbf{x}(t_i,x,\mathbf p),\mathbf{x}(t_i,y,\mathbf p)) 
\le \mu^i e^{- \kappa t_i} V_{p_1}(x,y).
\end{equation}
Then, from equations (\ref{eq:1}) and (\ref{eq:2}), for all $i\in \N$ and $t\in [t_{i},t_{i+1}]$, 
$$
V_{p_{i+1}}(\mathbf{x}(t,x,\mathbf p),\mathbf{x}(t,y,\mathbf p)) \le \mu^i e^{- \kappa t} V_{p_1}(x,y). 
$$
Since the switching signal $\mathbf{p}$ has dwell time $\tau_d$, it follows that $t_i \ge i \tau_d$ and therefore for all
$t\in [t_{i},t_{i+1}]$, $t \ge i \tau_d$. Since 
$\mu \geq 1 $, then for all $i\in \N$ and $t\in [t_{i},t_{i+1}]$,
$$
\mu^i = e^{i \log\mu} \le e^{\frac{\log \mu}{\tau_d} t}.
$$
Hence, for all $i\in \N$ and  $t\in [t_{i},t_{i+1}]$
$$
V_{p_{i+1}}(\mathbf{x}(t,x,\mathbf p),\mathbf{x}(t,y,\mathbf p)) \le e^{\left(\frac{\log \mu}{\tau_d}- \kappa\right) t} V_{p_1}(x,y) \le \overline{\alpha}(\|x-y\|)  e^{\left(\frac{\log \mu}{\tau_d}- \kappa\right) t}. 
$$
Therefore, for all $t \in \R^+_0$,
$$
\|\mathbf{x}(t,x,\mathbf p)-\mathbf{x}(t,y,\mathbf p)\| \le \underline{\alpha}^{-1} \left(\overline{\alpha}(\|x-y\|)  e^{\left(\frac{\log \mu}{\tau_d}- \kappa\right) t}\right).
$$
Equation (\ref{eq:stab2}) holds with the function $\beta$ given by 
$\beta(r,s)=\underline{\alpha}^{-1} \left(\overline{\alpha}(r) e^{\left(\frac{\log \mu}{\tau_d}- \kappa\right) s}\right)$  which belongs to class $\mathcal{KL}$ since by assumption $\frac{\log \mu}{\tau_d}- \kappa <0$. 
The same inequality can be shown for switching signals with a finite number of discontinuities; thus, $\Sigma$ is $\delta$-GUAS.
\end{proof}

In the following, we show that under the assumptions of Theorems \ref{th:stab1} or \ref{th:stab2}, ensuring incremental stability, it is possible to compute
approximately equivalent symbolic models of switched systems. We will make the following supplementary assumption on the $\delta$-GAS Lyapunov functions: for all $p\in P$, there exists a $\mathcal{K}_\infty$ function $\gamma_p$ such that 
\begin{equation}
\label{eq:assum}
\forall x,y,z \in \R^n, \;  |V_p(x,y)-V_p(x,z)| \le \gamma_p(\|y-z\|).
\end{equation}
Note that $\gamma_p$ is not a function of the variable $x$. It is convenient, for later use, to define the $\mathcal K_\infty$ function $\gamma$  by $\gamma = \max(\gamma_1,\dots ,\gamma_m)$.
We will discuss this assumption later in the paper and we will show that
it is not restrictive provided we are interested in the dynamics of the switched system on a compact subset of the state space $\R^n$.

\section{Approximate bisimulation}

In this section, we present a notion of approximate equivalence
which will relate a switched system to the symbolic models that we construct. We
start by introducing the class of transition systems which allows
us to model switched and symbolic systems in a common framework.
\begin{definition}
A transition system is a sextuple
$T=(Q,L,\rTo,O,H,I)$
consisting of:
\begin{itemize}
    \item a set of states $Q$;
    \item a set of labels $L$;
    \item a transition relation $\rTo\subseteq Q\times L \times Q$;
    \item an output set $O$;
    \item an output function $H:Q\rightarrow O$;
    \item a set of initial states $I \subseteq Q$.
\end{itemize}
$T$ is said to be \textit{metric} if the
output set $O$ is equipped with a metric ${d}$, \textit{countable}
if $Q$ and $L$ are countable sets, \textit{finite}, if $Q$ and $L$
are finite sets.
\end{definition}

The transition $(q,l,q')\in \rTo$ will
be denoted $q\rTo^{l} q'$ and means that the system can evolve from state $q$ to state $q'$ under the action
labelled by $l$. Thus, the transition relation captures the
dynamics of the transition system. 

Transition systems can serve as abstract models for describing
switched systems. Given a switched system $\Sigma=(\R^n,P, \mathcal
P, F)$ where $\mathcal P=\mathcal S(\R^+_0,P)$, we define the associated transition system 
$
T(\Sigma)=(Q,L,\rTo,O,H,I),
$
where the set of states is $Q=\R^n$; 
the set of labels is $L=P\times \R^+$;
the transition relation is given by 
$$
x\rTo^{p,\tau} x' \text{ if and only if } \mathbf{x}(\tau,x,p)=x',
$$ 
i.e. subsystem $\Sigma_p$ goes from state $x$ to state $x'$ in time $\tau$;
the set of outputs is $O=\R^n$; 
the observation map $H$ is the identity map over $\R^n$; 
the set of initial states is 
$I = \R^n$.
The transition system $T(\Sigma)$ is metric when the set of
outputs $O=\R^n$ is equipped with the metric $d(x,x')=\|x-x'\|$.
Note that the state space of $T(\Sigma)$ is infinite.

Usual equivalence relationships between transition systems rely
on the equality of observed behaviors. In this paper, we are mostly
interested in bisimulation equivalence~\cite{milner1989,park1981}.
Intuitively, a bisimulation relation between two transition
systems $T_1$ and $T_2$ is a relation between their set of states
explaining how a trajectory of $T_1$ can be transformed into
a trajectory of $T_2$ with the same associated sequence of outputs, and vice versa. 
The requirement of equality of output sequences, as in the
 classical formulation of bisimulation~\cite{milner1989,park1981} is quite strong for metric transition systems. We shall relax this, 
 by requiring output sequences to be close where closeness is measured with respect to the metric on the output space. This
relaxation leads to the notion of approximate bisimulation
relation introduced in~\cite{girard2007}.

\begin{definition}\label{def:bisim} Let $T_1=(Q_1,L,\rTo_1,O,H_1,I_1)$,
$T_2=(Q_2,L,\rTo_2,O,H_2,I_2)$ be metric transition systems with the
same sets of labels $L$ and outputs $O$ equipped with the 
metric $d$. Let $\varepsilon \in \R^+_0$ be a given precision, a
relation $R \subseteq Q_1\times Q_2$ is said to be an
$\varepsilon$-approximate bisimulation relation between $T_1$ and $T_2$
if for all $(q_1,q_2)\in R$:
\begin{itemize}
\item $d(H_1(q_1),H_2(q_2)) \le \varepsilon$;

\item for all $q_1\rTo_1^{l} q_1'$, there exists $q_2\rTo_2^l
q_2'$, such that $(q_1',q_2')\in R$;

\item for all $q_2\rTo_2^{l} q_2'$, there exists $q_1\rTo_1^l
q_1'$, such that $(q_1',q_2')\in R$.

\end{itemize}
The transition systems $T_1$ and $T_2$ are said to be
approximately bisimilar with precision $\varepsilon$, denoted
$T_1\sim_\varepsilon T_2$, if:
\begin{itemize}
\item for all $q_1\in I_1$, there exists $q_2 \in I_2$, such that
$(q_1,q_2) \in R$;

\item for all $q_2\in I_2$, there exists $q_1 \in I_1$, such that
$(q_1,q_2) \in R$.
\end{itemize}
\end{definition}

\section{Approximately bisimilar symbolic models}

In the following, we will work with a sub-transition system of
$T(\Sigma)$ obtained by selecting the transitions of $T(\Sigma)$
that describe trajectories of duration $\tau_s$ for some chosen
$\tau_s \in \R^+$. This can be seen as a sampling process. Particularly, we suppose that
switching instants can only occur at times of the form $i \tau_s$
with $i\in \N$.
This is a natural constraint when the switching in
$\Sigma$ has to be controlled by a  microprocessor with clock period
$\tau_s$.
Given a switched system $\Sigma=(\R^n,P,\mathcal P, F)$ where $\mathcal P=\mathcal S(\R^+_0,P)$, and
a time sampling parameter $\tau_s \in \R^+$, we define the associated transition
system
$T_{\tau_s}(\Sigma)=(Q_1,L_1,\rTo_1,O_1,H_1,I_1)$
where the set of states is $Q_1=\R^n$; the set of labels is $L_1=P$;
the transition relation is given by 
$$
x\rTo_1^p x' \text{ if and only if } \mathbf{x}(\tau_s,x,p)=x'; 
$$
the set of outputs is $O_1=\R^n$;
the observation map $H_1$ is the identity map over $\R^n$; the set of initial states is $I_1=\R^n$.
The transition system $T_{\tau_s}(\Sigma)$ is metric when the set of
outputs $O_1=\R^n$ is equipped with the metric $d(x,x')=\|x-x'\|$.

\subsection{Common Lyapunov function}

We first examine the simpler case when there exists a common $\delta$-GAS Lyapunov function $V$ for subsystems $\Sigma_1,\dots,\Sigma_m$. 
We start by approximating the set of states $Q_1=\R^n$ by the
lattice:
$$
[\R^n]_{\eta}=\left\{q\in \R^n \,\,\left| \; q_{i}
=k_{i}\frac{2\eta}{\sqrt{n}},\; k_{i}\in\mathbb{Z},\;
i=1,...,n\right.\right\},
$$
where $\eta \in \R^+$ is a state space discretization parameter. By
simple geometrical considerations, we can check that for all $x\in
\R^n$, there exists $q\in [\R^n]_{\eta}$ such that $\|x-q\| \le
\eta$.

Let us define the approximate transition system
$
T_{\tau_s,\eta}(\Sigma)=(Q_2,L_2,\rTo_2,O_2,H_2,I_2),
$
where the set of states is $Q_2=[\R^n]_{\eta}$; the set of labels remains the same $L_2=L_1=P$;
the transition relation is given by 
$$
q \rTo_2^p q' \text{ if and only if } \|\mathbf{x}(\tau_s,q,p)-q'\|\le \eta;
$$ 
the set of outputs remains the same $O_2=O_1=\R^n$; the observation map $H_2$ is the natural inclusion map from
$[\R^n]_{\eta}$ to $\R^n$, i.e. $H_2(q)=q$; the set of initial states is $I_2=[\R^n]_{\eta}$.
Note that the transition system $T_{\tau_s,\eta}(\Sigma)$ is
countable. Moreover, it is metric when the set of outputs
$O_2=\R^n$ is equipped with the metric $d(q,q')=\|q-q'\|$. An illustration of the approximation principle is shown on 
Figure~\ref{fig:approx}.

\begin{figure}[!h]
\begin{center}
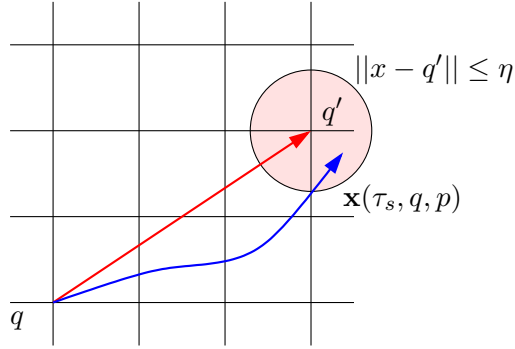
\end{center}
\caption{\label{fig:approx} Approximation principle for the computation of the symbolic model.}
\end{figure}

We now give the result that relates the existence of a common $\delta$-GAS Lyapunov function for the
subsystems $\Sigma_1,\dots,\Sigma_m$ to the existence of approximately bisimilar symbolic
models for the transition system $T_{\tau_s}(\Sigma)$.

\begin{theorem}\label{th:symb} Consider a switched system $\Sigma=(\R^n,P,\mathcal P,F)$ with $\mathcal P = \mathcal{S}(\R^+_0,P)$, time and state space sampling parameters $\tau_s, \eta\in \R^+$  and a desired
precision $\varepsilon \in \R^+$. Let us assume that there exists $V:\R^n\times \R^n \rightarrow
\R^+_0$ which is a common $\delta$-GAS Lyapunov function for subsystems
$\Sigma_1,\dots,\Sigma_m$ and such that equation (\ref{eq:assum}) holds for some $\mathcal{K}_{\infty}$ function $\gamma$. If 
\begin{equation}
\label{eq:simcond}
\eta \le \min\left\{ 
\gamma^{-1} \left((1-e^{-\kappa \tau_s}) \underline{\alpha}(\varepsilon)\right),
\overline{\alpha}^{-1} \left(  \underline{\alpha}(\varepsilon) \right)
\right\}
\end{equation}
then,
the transition systems $T_{\tau_s}(\Sigma)$ and
$T_{\tau_s,\eta}(\Sigma)$ are approximately bisimilar with precision
$\varepsilon$.
\end{theorem}

\begin{proof} We start by showing that the relation $R \subseteq
Q_1 \times Q_2$ defined by $(x,q)\in R$, if and only if $V(x,q) \le \underline{\alpha}(\varepsilon)$,
is an $\varepsilon$-approximate bisimulation relation. Let $(x,q)\in R$, then we have that
$
\|x-q\| \le \underline{\alpha}^{-1} \left(V(x,q) \right) \le \varepsilon.
$
Thus, the first condition of Definition~\ref{def:bisim} holds.
Let $x \rTo_1^p x'$, then $x'=\mathbf{x}(\tau_s,x,p)$. There exists
$q' \in [\R^n]_{\eta}$ such that
$\|\mathbf{x}(\tau_s,q,p)-q'\| \le \eta.$
Then, we have $q \rTo_2^p q'$. Let us check that $(x',q')\in R$. From equation (\ref{eq:assum}),
$$
|V(x',q') - V(x',\mathbf{x}(\tau_s,q,p))| \le \gamma(\|q' -\mathbf{x}(\tau_s,q,p))\|) \le \gamma(\eta).
$$
It follows that
\begin{eqnarray}
\nonumber
V(x',q') &\le&  V(x',\mathbf{x}(\tau_s,q,p)) + \gamma(\eta) = V(\mathbf{x}(\tau_s,x,p),\mathbf{x}(\tau_s,q,p))+ \gamma(\eta) \\
\label{eq:4}
& \le & e^{-\kappa \tau_s} V(x,q) + \gamma(\eta)
\end{eqnarray}
because $V$ is a $\delta$-GAS Lyapunov function for subsystem $\Sigma_p$. Then, from equation~(\ref{eq:simcond})
 and since $\gamma$
is a $\mathcal K_\infty$ function,
$$
V(x',q') \le e^{-\kappa \tau_s}\underline{\alpha}(\varepsilon) + \gamma(\eta)\le \underline{\alpha}(\varepsilon).
$$
Hence, $(x',q')\in R$.
In a similar way, we can prove
that, for all $q\rTo_2^p q'$, there is $x \rTo_1^p x'$ such that
$(x',q')\in R$. Hence $R$ is an $\varepsilon$-approximate bisimulation
relation between $T_{\tau}(\Sigma)$ and $T_{\tau,\eta}(\Sigma)$.

By definition of $I_2=[\R^n]_{\eta}$, for all $x \in I_1=\R^n$,
there exists $q \in I_2$ such that $\|x-q\|\le \eta$. 
Then, 
$$
V(x,q) \le \overline{\alpha}(\|x-q\|) \le \overline{\alpha}(\eta) \le \underline{\alpha}(\varepsilon)
$$
because of equation~(\ref{eq:simcond}) and $\overline{\alpha}$ is a $\mathcal K_\infty$ function. Hence, $(x,q)\in R$.
Conversely, for all $q \in I_2$, $x=q \in
\R^n=I_1$, then $V(x,q)=0$ and $(x,q)\in R$. Therefore, $T_{\tau_{s}}(\Sigma)$ and
$T_{\tau_{s},\eta}(\Sigma)$ are approximately bisimilar with precision
$\varepsilon$.
\end{proof}


Let us remark that, for a given time sampling parameter $\tau_s$ and a desired precision $\varepsilon\in \R^+$, there
always exists $\eta \in \R^+ $ sufficiently small such that equation (\ref{eq:simcond}) holds. This means that for switched systems admitting a common $\delta$-GAS Lyapunov function there exists approximately bisimilar symbolic models and any precision can be reached for all sampling rates.

The approach presented in this section for the computation of symbolic abstractions is quite similar to the approach presented in~\cite{pola2007} for $\delta$-GAS continuous control systems. Though, instead of defining the approximate bisimulation relation using the infinity norm as in~\cite{pola2007}, we use
sublevel sets of the common $\delta$-GAS Lyapunov function. This makes it possible, unlike in~\cite{pola2007}, to compute symbolic models for
arbitrary small time sampling parameter $\tau_s$. Further, this allows us to extend our approach to switched systems with multiple $\delta$-GAS Lyapunov functions.

\subsection{Multiple Lyapunov functions}

If a common $\delta$-GAS Lyapunov function does not exist, it remains possible to compute approximately bisimilar symbolic models provided we restrict the set of switching signals using a dwell time $\tau_d$. In this section, we consider a switched system $\Sigma_{\tau_d}=(\R^n,P, \mathcal P, F)$ where $\mathcal P= \mathcal S_{\tau_d}(\R_0^+,P)$. 
Let $\tau_s$ be a time sampling parameter; for simplicity and without loss of generality, we will assume that the dwell time $\tau_d$ is an integer multiple of $\tau_s$: there exists $N\in \N$ such that $\tau_d= N \tau_s$. Representing $\Sigma_{\tau_d}$ using a transition system is a bit less trivial than previously as we need to record inside the state of the transition system the time elapsed since the latest switching occurred. Thus, the transition system associated with $\Sigma_{\tau_d}$ is
$T_{\tau_s}(\Sigma_{\tau_d})=(Q_1,L_1,\rTo_1,O_1,H_1,I_1)$ 
where:
\vspace{-0.2cm}
\begin{itemize}
\item The set of states is $Q_1=\R^n \times P \times \{0,\dots,N-1\}$, a state $(x,p,i)\in Q_1$ means that the current state of $\Sigma_{\tau_d}$ is $x$, the current value of the switching signal is $p$ and the time elapsed since the latest switching is exactly $i\tau_s$, if $i<N-1$, or at least $(N-1)\tau_s$, if $i=N-1$.
\item The set of labels is $L_1=P$.
\item The transition relation is given by $(x,p,i) \rTo_1^l (x',p',i')$ if and only if $l=p$ and one the following holds:
\begin{itemize}
\item  $i<N-1$, $x'=\mathbf{x}(\tau_s,x,p)$, $p'=p$ and $i'=i+1$: switching is not allowed because the time elapsed since the latest switch is strictly smaller than the dwell time;
\item  $i=N-1$, $x'=\mathbf{x}(\tau_s,x,p)$, $p'=p$ and $i'=N-1$: switching is allowed but no switch occurs;
\item  $i=N-1$, $x'=\mathbf{x}(\tau_s,x,p)$, $p' \ne p$ and $i'=0$: switching is allowed and a switch occurs.
\end{itemize}
\item The set of outputs is $O_1=\R^n$.
\item The observation map $H_1$ is given by $H_1((x,p,i))=x$. 
\item The set of initial states is $I_1=\R^n \times P \times \{0\}$.
\end{itemize}
One can verify that the output trajectories of $T_{\tau_s}(\Sigma_{\tau_d})$ are the output trajectories of $T_{\tau_s}(\Sigma)$ associated with switching signals with dwell time $\tau_d=N\tau_s$. 
The approximation of the set of states of $T_{\tau_s}(\Sigma_{\tau_d})$ by a symbolic model is done using a lattice, as previously. Let $\eta \in \R^+$ be a state space discretization parameter, we define the transition system
$
T_{\tau_s,\eta}(\Sigma_{\tau_d})=(Q_2,L_2,\rTo_2,O_2,H_2,I_2)
$
where:
\vspace{-0.2cm}
\begin{itemize}
\item The set of states is $Q_2=[\R^n]_{\eta} \times P \times \{0,\dots,N-1\}$.
\item The set of labels remains the same $L_2=L_1=P$.
\item The transition relation is given by $(q,p,i) \rTo_2^l (q',p',i')$ if and only if $l=p$ and one of the following holds:
\begin{itemize}
\item  $i<N-1$, $\|\mathbf{x}(\tau_s,q,p)-q'\| \le \eta$, $p'=p$ and $i'=i+1$;
\item  $i=N-1$, $\|\mathbf{x}(\tau_s,q,p)-q'\| \le \eta$, $p'=p$ and $i'=N-1$;
\item  $i=N-1$, $\|\mathbf{x}(\tau_s,q,p)-q'\| \le \eta$, $p' \ne p$ and $i'=0$.
\end{itemize}
\item The set of outputs remains the same $O_2=O_1=\R^n$.
\item The observation map $H_2$ is given by $H_2((q,p,i))=q$. 
\item The set of initial states is $I_2=[\R^n]_{\eta} \times P \times \{0\}$.
\end{itemize}
Note that the transition system $T_{\tau_s,\eta}(\Sigma_{\tau_d})$ is
countable. Moreover, $T_{\tau_s}(\Sigma_{\tau_d})$ and $T_{\tau_s,\eta}(\Sigma_{\tau_d})$ are metric when the set of outputs
$O_1=O_2=\R^n$ is equipped with the metric $d(x,x')=\|x-x'\|$. The following theorem establishes the approximate
equivalence of $T_{\tau_s}(\Sigma_{\tau_d})$  and $T_{\tau_s,\eta}(\Sigma_{\tau_d})$.

\begin{theorem}\label{th:symb2} Consider $\tau_d \in \R^+_0$, 
a switched system $\Sigma_{\tau_d}=(\R^n,P,\mathcal P,F)$ with $\mathcal P = \mathcal{S}_{\tau_d}(\R^+_0,P)$, time and state space sampling parameters $\tau_s, \eta\in \R^+$  and a desired
precision $\varepsilon \in \R^+$. Let us assume that for all $p\in P$, there exists a $\delta$-GAS Lyapunov function $V_p$ for subsystem
$\Sigma_{\tau_d,p}$ and that equations (\ref{eq:reset}) and (\ref{eq:assum}) hold for some $\mu \ge 1$ and $\mathcal K_\infty$ functions $\gamma_1,\dots,\gamma_m$. 
If $\tau_d > \frac{\log\mu}{\kappa}$ and
\begin{equation}
\label{eq:simconda}
\eta \le \min\left\{ 
\gamma^{-1} \left(\frac{\frac{1}{\mu} - e^{-\kappa \tau_d}}{1-e^{-\kappa \tau_d}} (1-e^{-\kappa \tau_s}) \underline{\alpha}(\varepsilon)\right),
\overline{\alpha}^{-1} \left(  \underline{\alpha}(\varepsilon) \right)
\right\}
\end{equation}
then,
the transition systems $T_{\tau_s}(\Sigma_{\tau_d})$ and
$T_{\tau_s,\eta}(\Sigma_{\tau_d})$ are approximately bisimilar with precision
$\varepsilon$.
\end{theorem}

\begin{proof} Let us define the relation $R \subseteq Q_1 \times Q_2$ by
$$
R= \left\{(x,p_1,i_1,q,p_2,i_2) \in Q_1 \times Q_2 |\; p_1=p_2=p,\; i_1=i_2=i, V_p(x,q) \le \delta_i\} \right.
$$
where $\delta_0,\dots,\delta_{N}$ are given recursively by
$$
\delta_0=\underline{\alpha}({\varepsilon}), \; \delta_{i+1} = e^{-\kappa \tau_s} \delta_i+ \gamma(\eta).
$$
We can easily show that: 
\begin{equation}
\label{eq:3}
\delta_i =e^{-i\kappa \tau_s}\underline{\alpha}({\varepsilon}) + \gamma(\eta)\frac{1-e^{-i\kappa \tau_s}}{1-e^{-\kappa \tau_s}} = \frac{\gamma(\eta)}{1-e^{-\kappa \tau_s}} +  e^{-i\kappa \tau_s}\left(
\underline{\alpha}({\varepsilon})-\frac{\gamma(\eta)}{1-e^{-\kappa \tau_s}} \right)
\end{equation}
From equation (\ref{eq:simconda}) and since $\mu \ge 1$ and $\gamma$ is a $\mathcal K_\infty$ function, 
$\gamma(\eta) \le (1-e^{-\kappa \tau_s}) \underline{\alpha}({\varepsilon})$. It follows from (\ref{eq:3}) that
$\delta_0 \ge \delta_1 \ge \dots \ge \delta_{N-1} \ge \delta_N$. 
From equation (\ref{eq:simconda}), and since $\gamma$ is a $\mathcal K_\infty$ function  and $\tau_d=N\tau_s$,
$$
\delta_N = e^{-\kappa \tau_d} \underline{\alpha}({\varepsilon}) + \gamma(\eta) \frac{1-e^{-\kappa \tau_d}}{1-e^{-\kappa \tau_s}} \le e^{-\kappa \tau_d} \underline{\alpha}({\varepsilon})+ \left(\frac{1}{\mu} -e^{-\kappa \tau_d}\right) \underline{\alpha}({\varepsilon})  =\frac{ \underline{\alpha}({\varepsilon})}{\mu}.
$$
We can now prove that $R$ is an $\varepsilon$-approximate bisimulation relation between $T_{\tau_s}(\Sigma_{\tau_d})$ and
$T_{\tau_s,\eta}(\Sigma_{\tau_d})$. Let $(x,p,i,q,p,i)\in R$,
then 
\begin{eqnarray*}
\|H_1(x,p,i)-H_2(q,p,i)\| &=& \|x-q\| \le \underline{\alpha}^{-1} \left( V_p(x,q) \right)\\ &\le& 
 \underline{\alpha}^{-1}(\delta_i) \le  \underline{\alpha}^{-1}(\delta_0) = \varepsilon.
\end{eqnarray*}
Hence, the first condition of Definition~\ref{def:bisim} holds. Let us prove that the second condition holds as well.
Let $(x,p,i)\rTo_1^p (x',p',i')$, then $x'= \mathbf{x}(\tau_s,x,p)$. There exists a transition $(q,p,i)\rTo_2^p (q',p',i')$ with $\|q'-\mathbf{x}(\tau_s,q,p)\| \le \eta$. 
From equation (\ref{eq:assum}) and since $V_p$ is a $\delta$-GAS Lyapunov function for subsystem $\Sigma_p$ we can show, similarly to equation (\ref{eq:4}), that
\begin{eqnarray}
\label{eq:5}
V_p(x',q') \le e^{-\kappa \tau_s} V_p(x,q) +\gamma(\eta)  \le e^{-\kappa \tau_s} \delta_i +\gamma(\eta) = \delta_{i+1}.
\end{eqnarray}
We now examine three separate cases:
\begin{itemize}
\item If $i<N-1$, then $p'=p$ and $i'=i+1$; since $V_p(x',q') \le \delta_{i+1}$, it follows that $(x',p,i+1,q',p,i+1) \in R$.
\item If $i=N-1$ and $p'=p$, then $i'=N-1$; from (\ref{eq:5}), $V_p(x',q') \le \delta_{N} \le \delta_{N-1}$, it follows that $(x',p,N-1,q',p,N-1) \in R$.
\item If $i=N-1$ and $p'\ne p$, then $i'=0$; from (\ref{eq:5}), $V_p(x',q') \le \delta_{N} \le \delta_{0}/\mu$. From equation (\ref{eq:assum}), it follows that $V_{p'}(x',q') \le \mu V_p(x',q') \le \delta_0$. Therefore,
$(x',p',0,q',p',0) \in R$.
\end{itemize}
Similarly, we can show that for any transition $(q,p,i)\rTo_2^p (q',p',i')$, there exists a transition $(x,p,i)\rTo_1^p (x',p',i')$ such that $(x',p',i',q',p',i')\in R$. Hence, $R$ is an $\varepsilon$-approximate bisimulation relation.

For all initial states $(x,p,0)\in I_1$, there exists $(q,p,0) \in I_2$ such that $\|x-q\|\le \eta$. 
Then, $V_p(x,q) \le \overline{\alpha} (\eta) \le \underline{\alpha}(\varepsilon)$
because of equation~(\ref{eq:simconda}) and $\overline{\alpha}$ is $\mathcal K_\infty$ function. Hence, $V_p(x,q)  \le \delta_0$ and $(x,p,0,q,p,0)\in R$.
Conversely, for all $(q,p,0) \in I_2$, $(x,p,0)=(q,p,0)\in I_1$. Then, $V_p(x,q)=0 \le \delta_0$
and $(x,p,0,q,p,0)\in R$. Thus, $T_{\tau_s}(\Sigma_{\tau_d})$ and
$T_{\tau_s,\eta}(\Sigma_{\tau_d})$ are approximately bisimilar with precision
$\varepsilon$. 
\end{proof}

Provided that $\tau_d > \frac{\log\mu}{\kappa}$, for a given time sampling parameter and a desired precision, there always exists $\eta \in \R^+$ sufficiently small such that equation (\ref{eq:simconda}) holds.
Thus, if the dwell time is large enough, we can compute symbolic models of arbitrary precision of the switched system.
Let us remark that the lower bound we obtain on the dwell time is the same than the one in Theorem~\ref{th:stab2} ensuring incremental stability of the switched system. Also, Theorem~\ref{th:symb} can be seen as a corollary of Theorem~\ref{th:symb2}. Indeed, existence of a common $\delta$-GAS Lyapunov function is equivalent to equation (\ref{eq:reset}) with $\mu=1$. Then, no constraint is necessary on the dwell time and equation (\ref{eq:simconda}) becomes equivalent to (\ref{eq:simcond}). 

The previous Theorems also give indications on the practical computation of these
symbolic models. The sets of states of $T_{\tau_s,\eta}(\Sigma)$ or
$T_{\tau_s,\eta}(\Sigma_{\tau_d})$ are countable but infinite. 
However, in practical control applications, we are usually interested in the dynamics of the switched system only on a compact
subset  $C \subseteq \R^n$. Then, we can restrict the set
of states of $T_{\tau_s,\eta}(\Sigma)$ or $T_{\tau_s,\eta}(\Sigma_{\tau_d})$ 
to the sets $[\R^n]_{\mu} \cap  C$ or $([\R^n]_{\mu} \cap  C )\times P \times \{0,\dots,N-1\}$ which are finite.
The computation of the transition relations is then relatively simple since
it mainly involves the numerical computation of the points $\mathbf{x}(\tau_s,q,p)$ with $q \in [\R^n]_{\mu} \cap  C$
and $p \in P$. This can be done by simulation of 
the subsystems $\Sigma_1,\dots,\Sigma_m$. Numerical errors in the computation of these points can be taken
into account: it is sufficient to replace $\eta$ by $\eta + e$, where $e$ is an evaluation of the error, in Theorems~\ref{th:symb} and \ref{th:symb2}.

Finally, we would like to discuss the assumption made in equation (\ref{eq:assum}). This assumption may look quite strong
because the inequality has to hold for any triple in $\R^n$, and the function $\gamma_p$ must be independent of $x$. However, if we are interested in the dynamics of the switched system on the compact
subset  $C \subseteq \R^n$, we only need this assumption to hold for all $x,y,z\in C$. 
Then, it is sufficient to assume that $V_p$ is $\mathcal C^1$ on $ C$. Indeed, for all $x,y,z \in C$,
$$
|V_p(x,y)-V_p(x,z) | \le \left( \max_{x,y \in C} \left\|\frac{\partial V_p}{\partial y}(x,y) \right\| \right) \|y-z\| =\gamma_p(\|y-z\|).
$$
In this case, equation (\ref{eq:assum}) holds. This means that the existence of approximately bisimilar symbolic models on an arbitrary compact subset of $\R^n$ does not need more assumptions than existence of common or multiple Lyapunov functions
ensuring incremental stability of the switched system.

\section{Examples of symbolic control design}

In this section, we show the effectiveness of our approach on two examples illustrating the main results of the paper.

\subsection{Common Lyapunov functions: the boost DC-DC converter}

We first use our methodology to compute symbolic models of a concrete switched system: 
the boost DC-DC converter (see Figure~\ref{fig:dc}). This is an example of electrical power convertor
that has been studied from the point of view of hybrid control in~\cite{senesky2003,beccuti2005,buisson2005,beccuti2006}.
\begin{figure}[!h]
\begin{center}
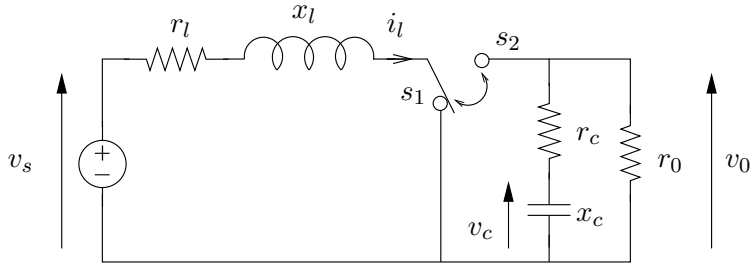
\caption{boost DC-DC converter.}
\label{fig:dc}
\end{center}
\end{figure}

\noindent
The boost converter has two operation modes depending on the position of the switch. 
The state of the system is $x(t)=[i_l(t) \; v_c(t)]^T$ where $i_l(t)$ is the inductor current and $v_c(t)$ the capacitor
voltage.
The dynamics associated with both modes are affine of the form
$
\dot x(t)=A_p x(t) +b$ ($p=1,2$)
with 
$$
A_1=\left[ 
\begin{smallmatrix}
-\frac{r_l}{x_l} & 0 \\
0  & -\frac{1}{x_c}\frac{1}{r_0+r_c}
\end{smallmatrix}
\right],
\;
A_2=\left[ 
\begin{smallmatrix}
-\frac{1}{x_l}(r_l+\frac{r_0r_c}{r_0+r_c}) & -\frac{1}{x_l}\frac{r_0}{r_0+r_c} \\
\frac{1}{x_c}\frac{r_0}{r_0+r_c}  & -\frac{1}{x_c}\frac{1}{r_0+r_c}
\end{smallmatrix}
\right],
\; b=\left[ 
\begin{smallmatrix}
\frac{v_s}{x_l} \\
0
\end{smallmatrix}
\right].
$$
It is clear that the boost DC-DC converter is an example of a switched system. In the following,
we use the numerical values from~\cite{beccuti2005}, that is, in the per unit system, $x_c=70$ p.u.,
$x_l=3$ p.u., $r_c=0.005$ p.u., $r_l=0.05$ p.u., $r_0=1$ p.u. and $v_s=1$ p.u.. The goal of the boost DC-DC converter
is to regulate the output voltage across the load $r_0$. This control problem is usually reformulated 
as a current reference scheme. Then, the goal is to keep the inductor current $i_l(t)$ around a reference value $i_l^\text{ref}$. This can be done, for instance, by synthesizing a controller that keeps the state of the switched system in an invariant set $\mathcal I$ centered around the reference value.

It can be shown by solving a set of 2 linear matrix inequalities
that the subsystems associated with the two operation modes are both incrementally stable and that they
share a common $\delta$-GAS Lyapunov function of the form 
$
V(x,y) = \sqrt{(x-y)^T M (x-y)},
$
where  $M$ is positive definite symmetric. Thus, the switched system is $\delta$-GAS however it is not GAS
because the subsystems do not share a common equilibrium point.

The matrix $M$ can be computed using semi-definite programming; 
for a better numerical conditioning, we rescaled the second 
variable of the system (i.e.~the state of the system becomes $x(t)=[i_l(t) \; 5 v_c(t)]^T$; the matrices $A_1$, $A_2$
and vector $b$ are modified accordingly). We obtained 
$$
M= \left[ 
\begin{smallmatrix}
1.0224  &  0.0084 \\
0.0084  &  1.0031
\end{smallmatrix}
\right].
$$
The corresponding $\delta$-GAS Lyapunov function has the following characteristics:
$\underline{\alpha}(s)=s$, $\overline{\alpha}(s)=1.0127 s$, $\kappa = 0.014$. Let us remark that equation (\ref{eq:assum})
holds on the entire state-space with $\gamma(s)=1.0127 s$. 
We set the sampling period to $\tau_s=0.5$. 
Then, a symbolic model can be computed for the boost DC-DC converter using the procedure described in Section 4.1.
According to Theorem~\ref{th:symb}, a desired precision $\varepsilon$ can be achieved by choosing a state space discretization parameter $\eta$ satisfying $\eta \le \varepsilon/ 145$. In this example, the ratio between the precision of the symbolic approximation and the 
state space discretization parameter is quite large. This is explained by the fact that the subsystems are quite weakly stable since the value of $\kappa$ is small.

\begin{figure}[!h]
\vspace{-3cm}
\begin{center}
\includegraphics[angle=0,scale=0.4]{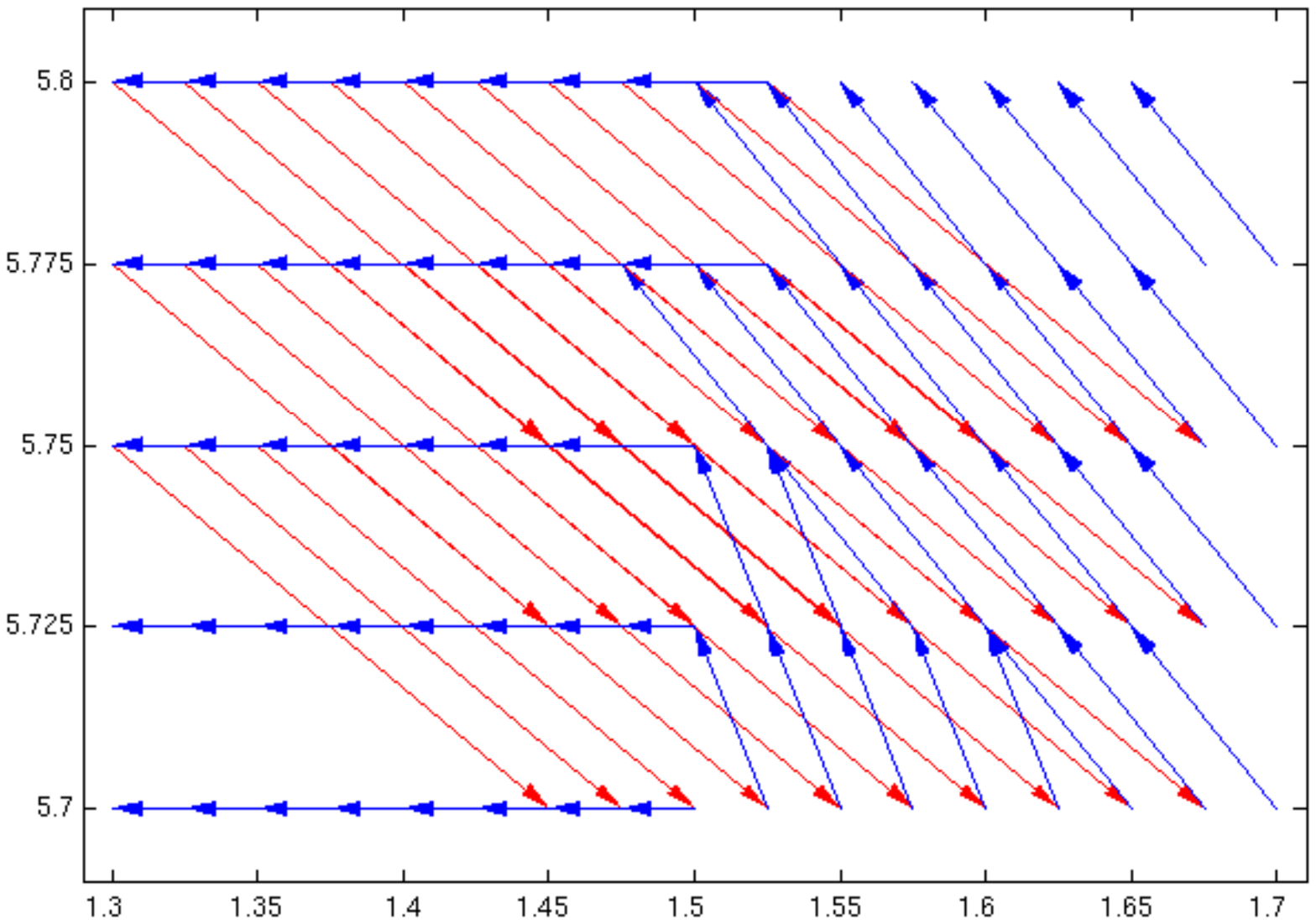}
\hspace{-1cm}
\includegraphics[angle=0,scale=0.4]{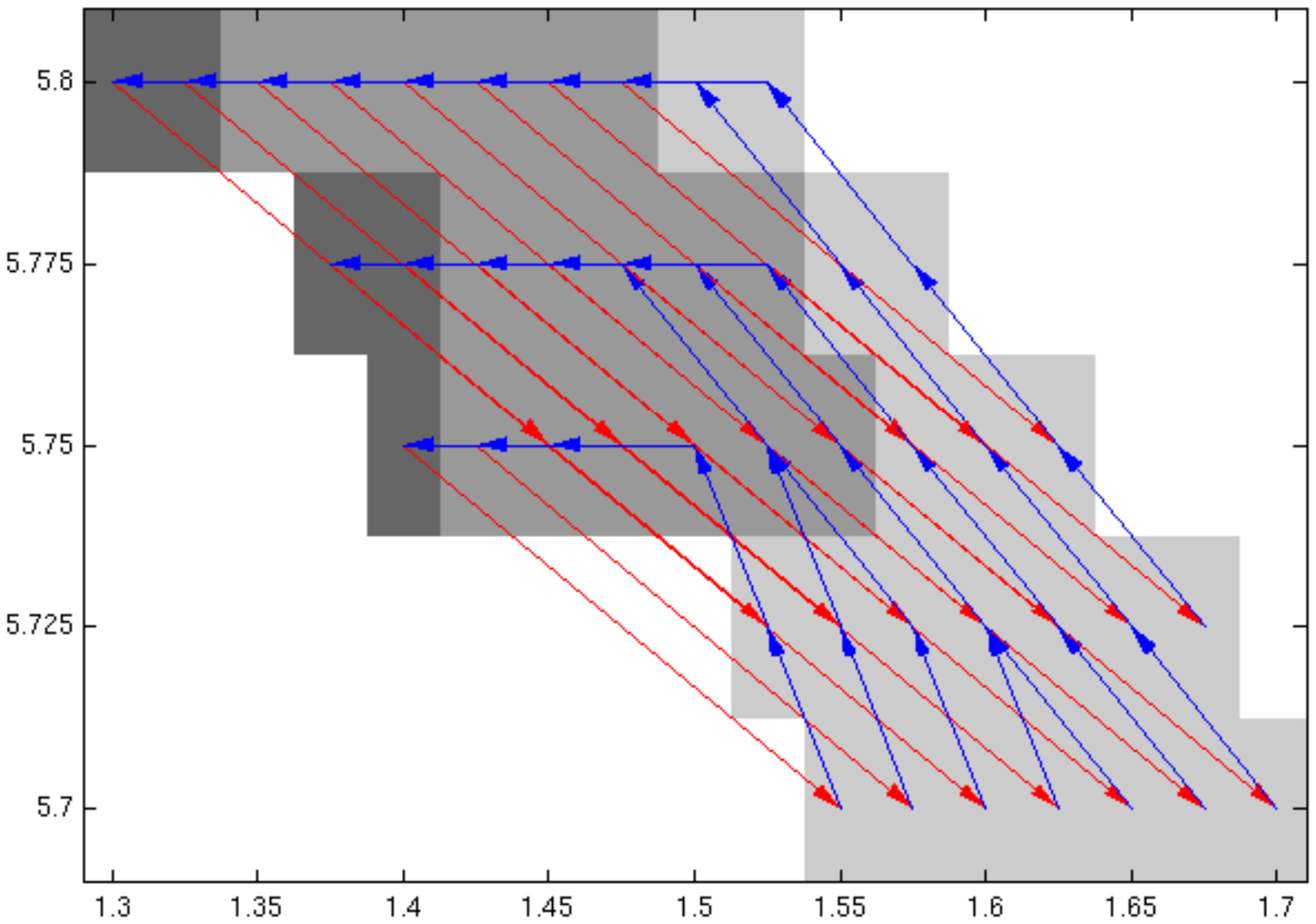}
\vspace{-3.2cm}
\caption{Symbolic model of the boost DC-DC converter for $\eta=\frac{1}{40\sqrt{2}}$ (left); Controller for the symbolic model (right) (dark gray: mode 1, light gray: mode 2, medium gray: both modes are acceptable, white: uncontrollable states).}
\label{fig:ex1}
\end{center}
\end{figure}

We consider two different values of the precision parameter $\varepsilon$. 
We first choose a precision $\varepsilon=2.6$ which can be achieved by choosing $\eta=\frac{1}{40\sqrt{2}}$.
This precision is quite poor and makes the computed symbolic model of no practical use. However, it helps to understand the second experiment decribed further. On Figure~\ref{fig:ex1}, 
the symbolic model of the boost DC-DC converter is shown on the left, red and blue arrows represent the transitions associated with mode 1 and 2, respectively. We only represented the transitions that keep the state of the symbolic model
in the set $\mathcal I'= [1.3, 1.7]\times [5.7,5.8]$. 
Using supervisory control~\cite{ramadge1987}, we synthesized the most permissive controller that keeps the state of the symbolic model inside $\mathcal I'$. It is shown on the right figure, the color of the boxes centered around the states of the symbolic model gives the corresponding action of the controller: dark and light gray means that for these states of the symbolic model
the controller has to use mode 1 and 2, respectively; medium gray means that for these states 
the controller can use either mode 1 or mode 2; white means that from these states there does not exist any switching sequence that keeps the state of the symbolic model in $\mathcal I'$, i.e. these states are uncontrollable. From this controller, using the approach presented 
in~\cite{tabuada2007}, one can derive a controller for the boost DC-DC converter that keeps the state of the switched system in 
$\mathcal I= [1.3-\varepsilon, 1.7+\varepsilon ]\times [5.7-\varepsilon,5.8+\varepsilon]$. Clearly, the chosen precision is too large to make this controller useful from a practical point of view.

The second value we consider for the precision parameter is $\varepsilon=0.026$. This precision can be achieved by 
choosing $\eta=\frac{1}{4000\sqrt{2}}$. We do not show the symbolic model as it has too many states ($642001$)
to be represented graphically.
We repeat the same experiment with this model, the most permissive controller that keeps the state of the symbolic model in $\mathcal I'$ is shown in Figure~\ref{fig:ex2}, on the left. The computation of the symbolic model and the synthesis of the supervisory controller, implemented in MATLAB, takes overall less than 60 seconds. From the controller of the symbolic model, we derive a controller for the boost DC-DC converter that keeps the state of the switched system in 
$\mathcal I= [1.3-\varepsilon, 1.7+\varepsilon ]\times [5.7-\varepsilon,5.8+\varepsilon]$. We apply a lazy control strategy, when the controller can choose both modes 1 and 2, it just keeps the current operation mode unchanged.
A state trajectory of the controlled boost DC-DC converter is shown in Figure~\ref{fig:ex2}, on the right. 
We can see that the trajectory remains in the invariant set.

\begin{figure}[!h]
\vspace{-3cm}
\begin{center}
\includegraphics[angle=0,scale=0.4]{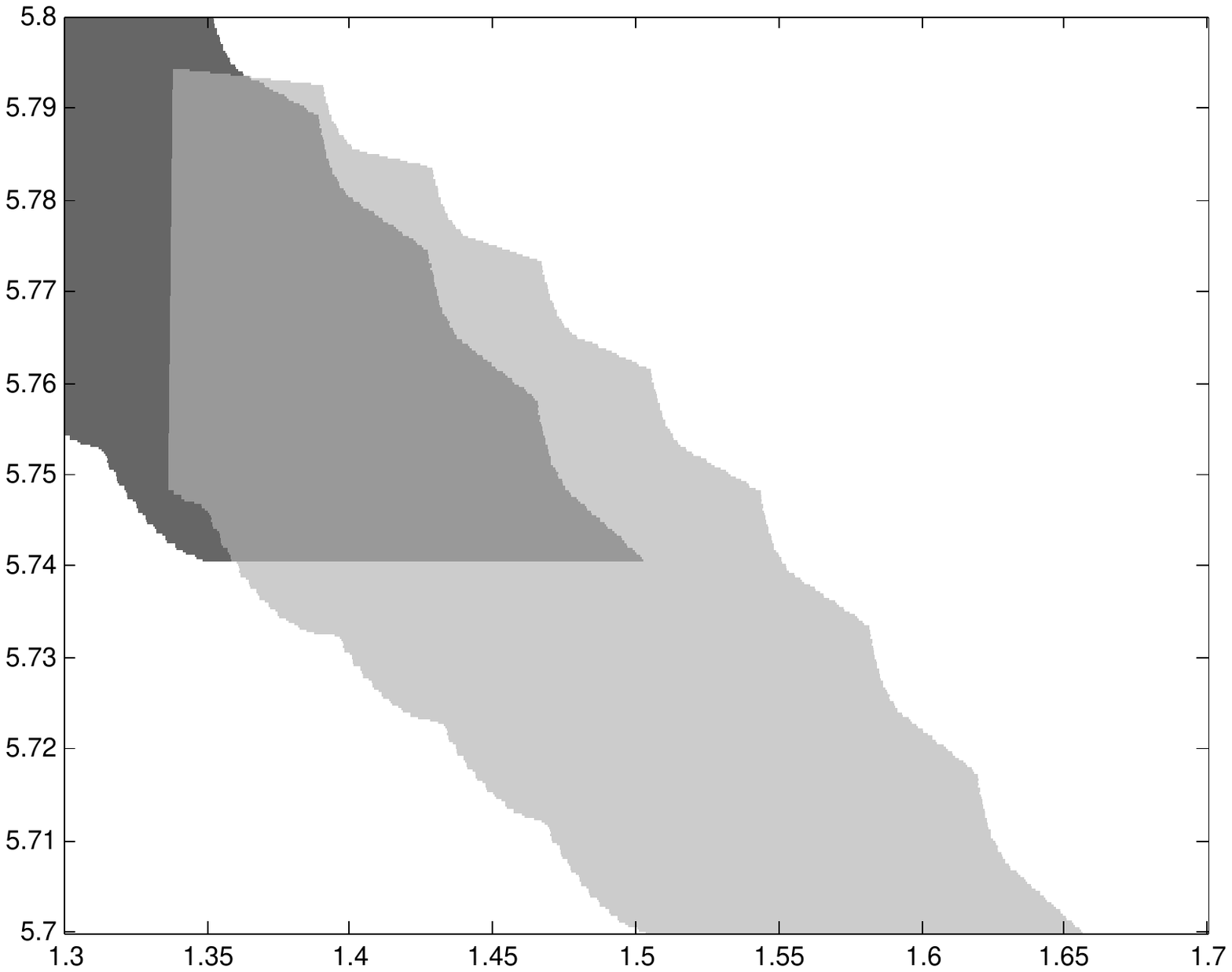}
\hspace{-1cm}
\includegraphics[angle=0,scale=0.4]{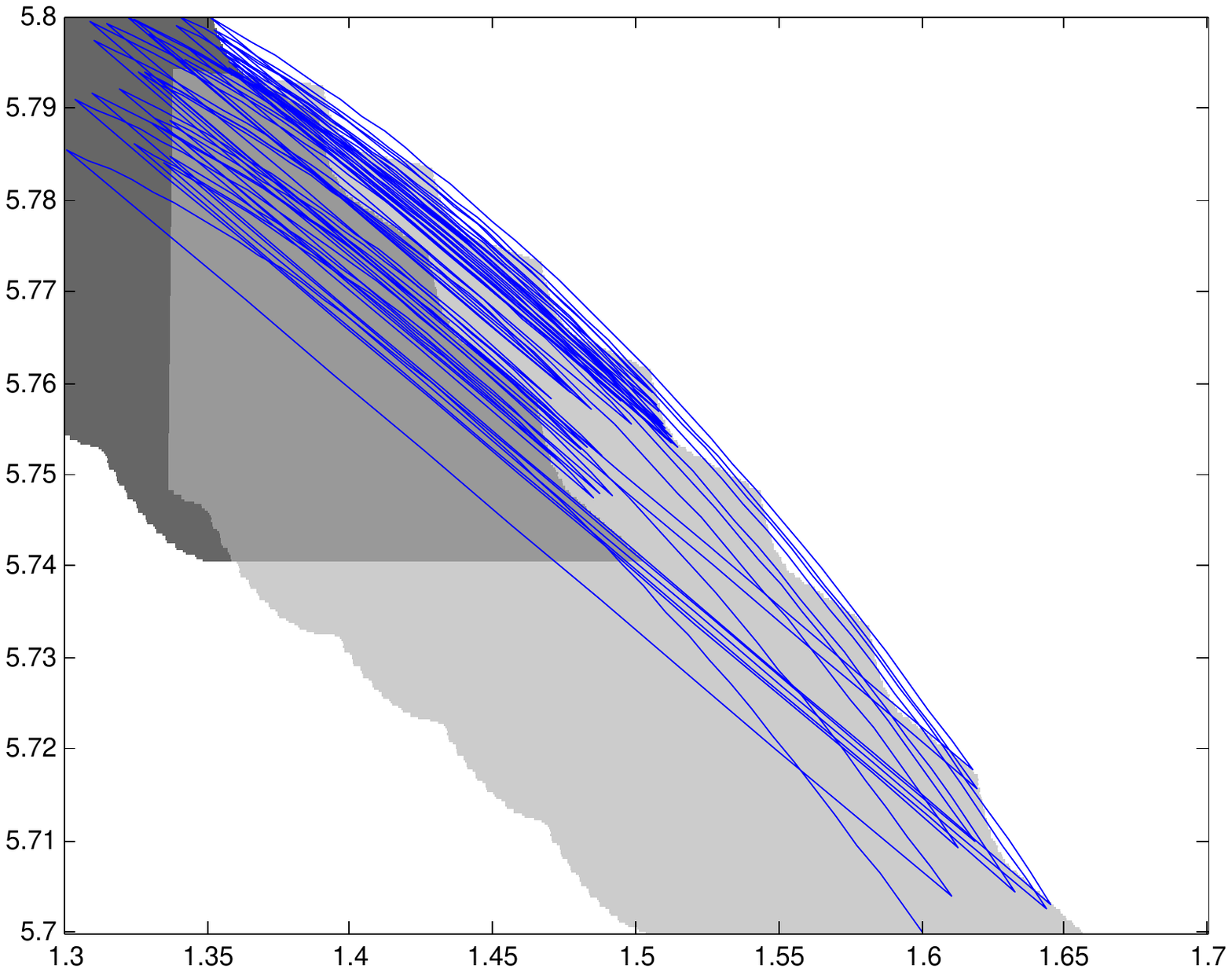}
\vspace{-3.2cm}
\caption{Controller for a symbolic model of the boost DC-DC converter for $\eta=\frac{1}{4000\sqrt{2}}$ (left)
(dark gray: mode 1, light gray: mode 2, medium gray: both modes are acceptable, white: uncontrollable states); Trajectory of the boost DC-DC converter using the previous controller (right).}
\label{fig:ex2}
\end{center}
\end{figure}

\subsection{Multiple Lyapunov functions} We now consider a second example inspired by a well known switched system with stable subsystems and exhibiting unstable behaviors (see e.g.~\cite{liberzon2003}). The system has two modes and the state space is $\R^2$. The dynamics associated with both modes are affine of the form 
$
\dot x(t)=A_p x(t) +b_p$ ($p=1,2$)
with 
$$
A_1=\left[ 
\begin{smallmatrix}
-0.25 & 1 \\
-2  & -0.25
\end{smallmatrix}
\right],
\;
A_2=\left[ 
\begin{smallmatrix}
-0.25 & 2 \\
-1  & -0.25
\end{smallmatrix}
\right],
\; b_1=\left[ 
\begin{smallmatrix}
-0.25 \\
-2
\end{smallmatrix}
\right],\;
b_2=\left[ 
\begin{smallmatrix}
0.25 \\
1
\end{smallmatrix}
\right].
$$
We consider a control design problem with a safety specification: the goal is to keep the trajectories of the switched
system within a specified region of the state-space, denoted $\mathcal I$, while avoiding a specified subset of unsafe states $\mathcal U \subseteq \mathcal I$. We assume that $\mathcal U$ contains the equilibrium points of both systems and therefore the specification cannot be met by neither $\Sigma_1$ nor $\Sigma_2$.

The system does not have a common $\delta$-GAS Lyapunov function because it exhibits unstable behaviors for some
switching signals (e.g. apply periodically mode $1$ during $1$ time unit, then mode $2$ during $1$ time unit and so on). However,
each subsystem has a $\delta$-GAS Lyapunov function of the form $
V_p(x,y) = \sqrt{(x-y)^T M_p (x-y)},$ with
$$
M_1=\left[ 
\begin{smallmatrix}
2 & 0 \\
0  & 1
\end{smallmatrix}
\right],
\;
M_2=\left[ 
\begin{smallmatrix}
1 & 0 \\
0  & 2
\end{smallmatrix}
\right].
$$
The corresponding $\delta$-GAS Lyapunov functions have the following characteristics:
$\underline{\alpha}(s)=s$, $\overline{\alpha}(s)=\sqrt{2} s$, $\kappa = 0.25$.
Moreover, the assumptions of Theorem~\ref{th:stab2} hold with $\mu=\sqrt{2}$, and a dwell-time $\tau_d=2 > \frac{\log(\mu)}{\kappa}$. Here, again, the switched system is $\delta$-GAS however it is not GAS
because the subsystems do not share a common equilibrium point. Also, equation (\ref{eq:assum})
holds on the entire state-space with $\gamma(s)=\sqrt{2} s$. 
We set the sampling period to $\tau_s=0.5$. 
Then, a symbolic model can be computed using the procedure described in Section 4.2.
According to Theorem~\ref{th:symb2}, a desired precision $\varepsilon$ can be achieved by choosing a state space discretization parameter $\eta$ satisfying $\eta \le \varepsilon/ 48$. 

We choose $\eta=\frac{1}{100\sqrt{2}}$, corresponding to a precision $\varepsilon = 0.34$. 
Then, we used supervisory control to design the most permissive controller that keeps the state of the symbolic model 
within the set $\mathcal I'= [-6, 6]\times [-4,4]$ while avoiding $\mathcal U' = [-1.5, 1.5]\times [-1,1]$.
Though our symbolic model has $7696008$ states, the overall computation, including the determination of the symbolic 
model and the synthesis of the controller, takes only about $130$ seconds.

The controller is shown on Figure~\ref{fig:ex3}. On the left, respectively on the right, we represented the possible control actions when the current mode is $1$, respectively $2$, and the dwell time has elapsed (i.e. switching is enabled). 
Dark and light gray means that for these states of the symbolic model
the controller has to use mode 1 and 2, respectively; medium gray means that for these states 
the controller can use either mode 1 or mode 2; white means that these states are uncontrollable and the specification
cannot be met from these states.

\begin{figure}[!h]
\vspace{-3cm}
\begin{center}
\includegraphics[angle=0,scale=0.4]{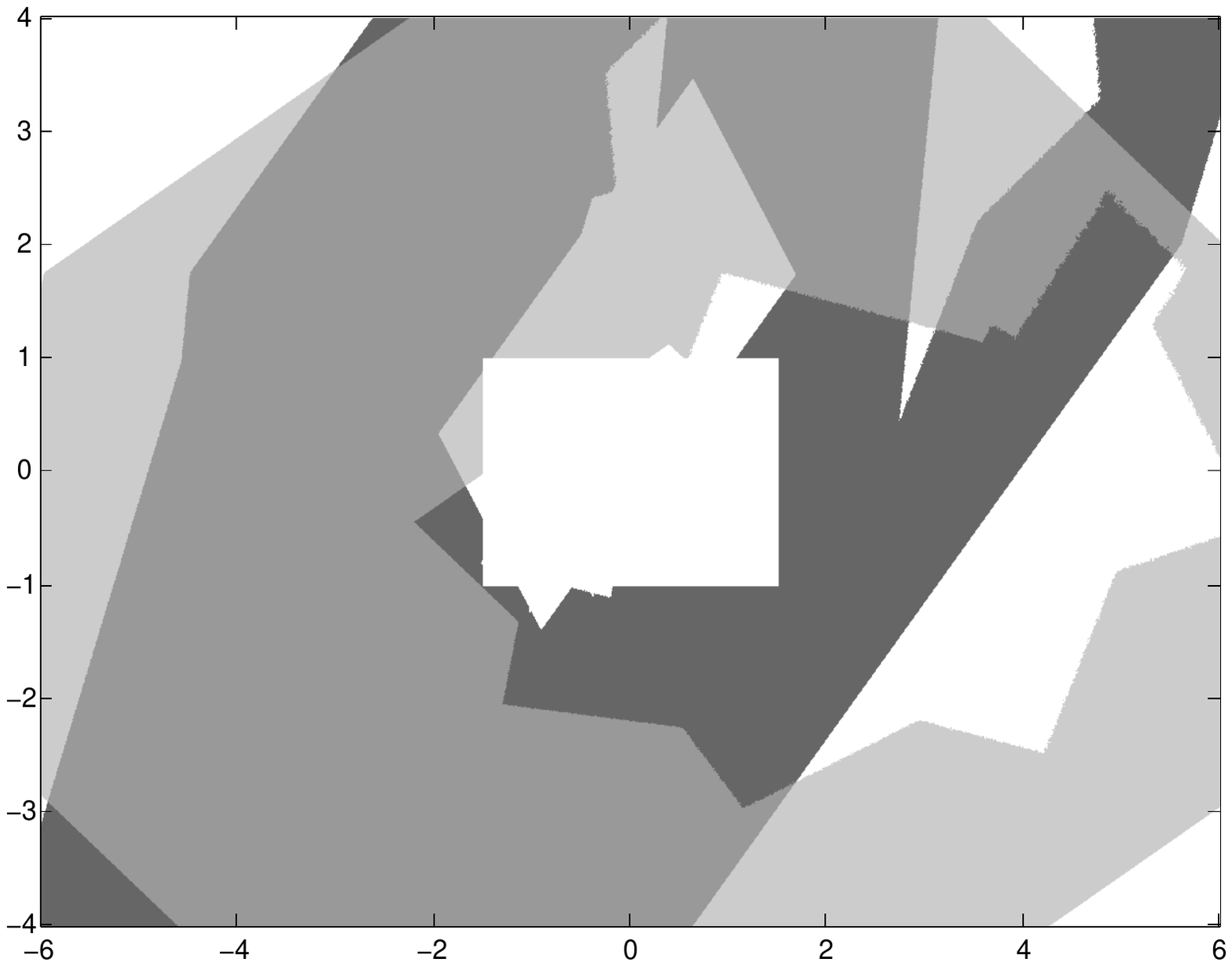}
\hspace{-1cm}
\includegraphics[angle=0,scale=0.4]{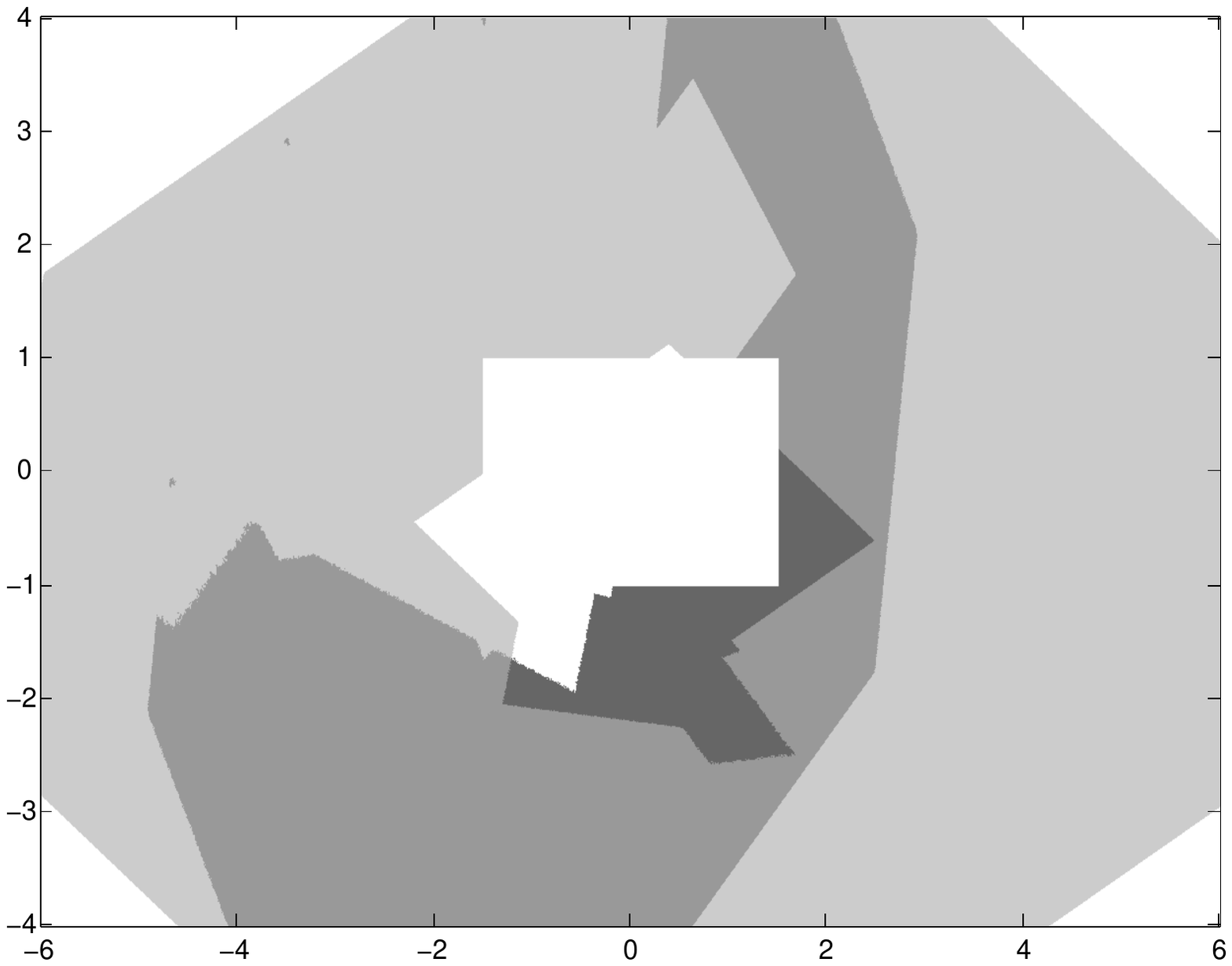}
\vspace{-3.2cm}
\caption{Controller for the symbolic model. Possible control actions when the current mode is $1$, respectively $2$, and the
dwell time has elapsed (left, respectively right) (dark gray: mode 1, light gray: mode 2, medium gray: both modes are acceptable, white: uncontrollable states).
\label{fig:ex3}}
\end{center}
\end{figure}

From the most permissive controller represented on Figure~\ref{fig:ex3}, we designed a lazy controller for the symbolic model. Unlike the most permissive controller, the lazy control strategy can be implemented regardless of the current mode
and of the time elapsed since the latest switching. The controller is represented
on Figure~\ref{fig:ex4}, on the left:  
dark and light gray means that for these states of the symbolic model
the controller has to use mode 1 and 2, respectively; medium gray means that for these states 
the controller must keep the current mode unchanged; white means that these states are uncontrollable.
Let us remark that by design, this controller satisfies the dwell time constraint though it does not appear explicitely in the controller description. 
Using the approach presented 
in~\cite{tabuada2007}, one can derive a controller for the switched system, that keeps the state of the switched system 
within the set 
$\mathcal I= [-6-\varepsilon, 6+\varepsilon]\times [-4-\varepsilon,4+\varepsilon]$ while avoiding $\mathcal U = [-1.5+\varepsilon, 1.5-\varepsilon]\times [-1+\varepsilon,1-\varepsilon]$.
On Figure~\ref{fig:ex4}, in the center, we represented an example of switching signals generated by the controller and the corresponding evolution of the state variables. We can check that the switching signal indeed has dwell time $2$.
On the right, we represented the associated trajectory of the switched system, satifying the safety property.

\begin{figure}[!h]
\vspace{-2cm}
\begin{center}
\includegraphics[angle=0,scale=0.29]{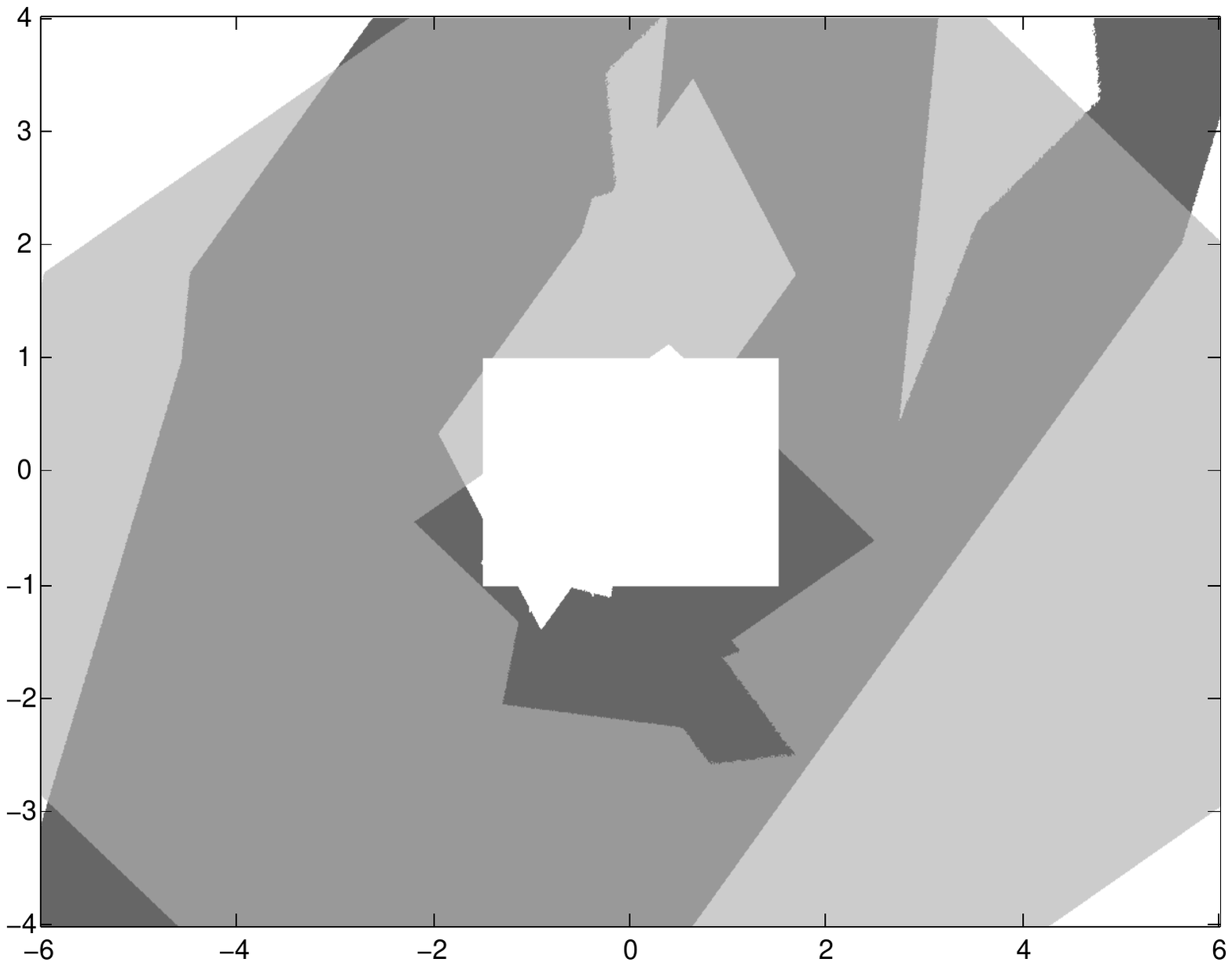}
\hspace{-1cm}
\includegraphics[angle=0,scale=0.29]{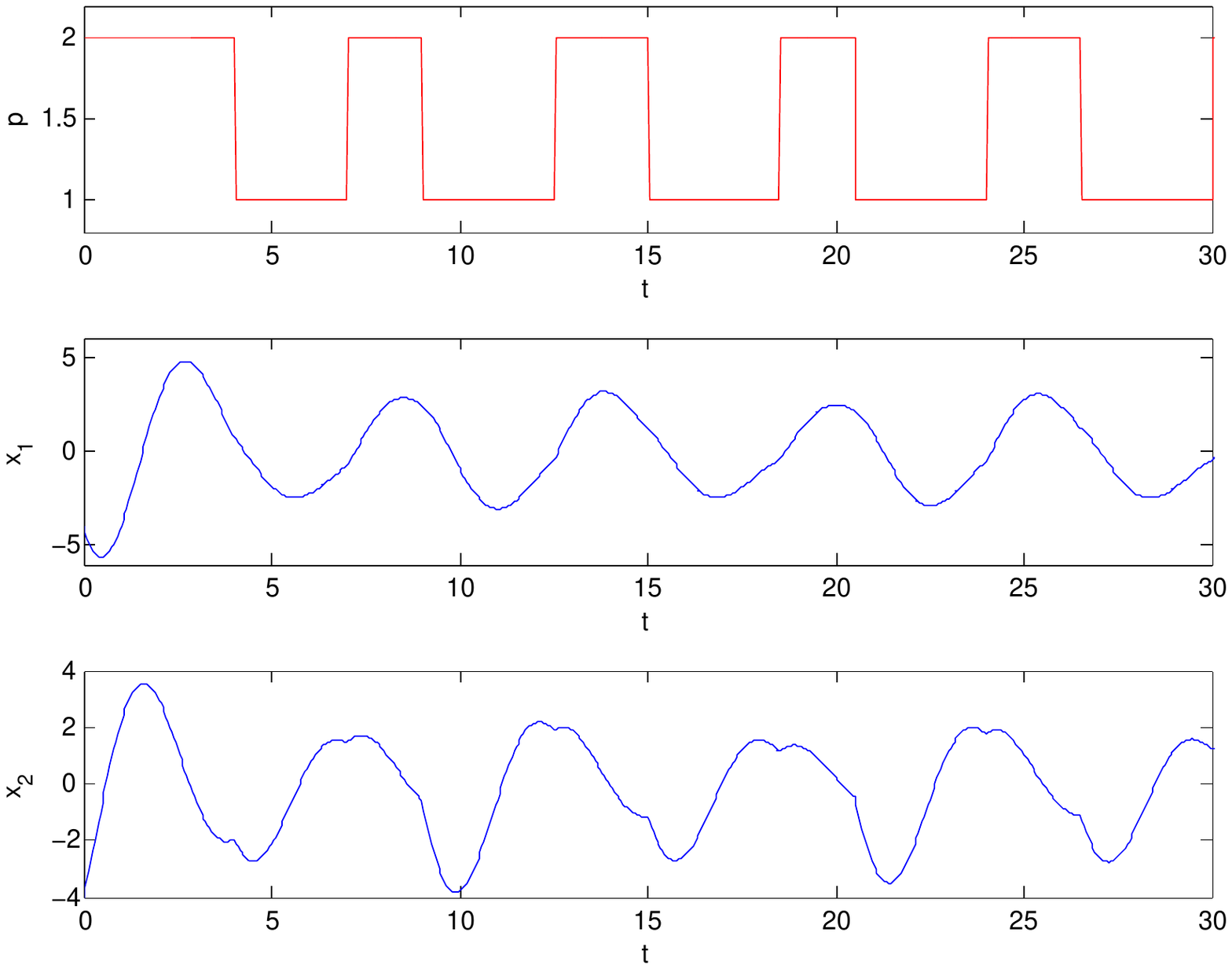}
\hspace{-1cm}
\includegraphics[angle=0,scale=0.29]{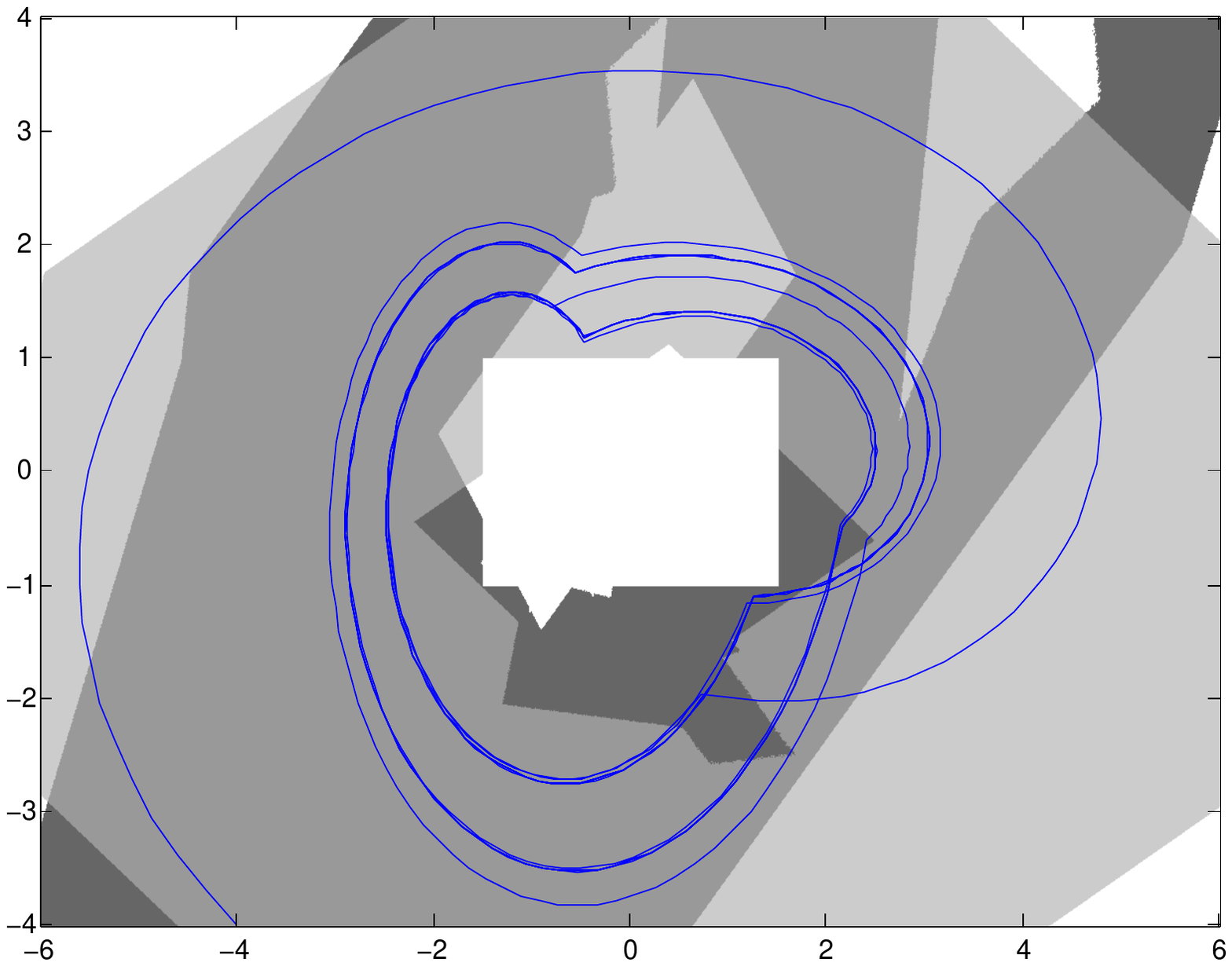}
\vspace{-3cm}
\caption{Lazy controller for the symbolic model: (left) (dark gray: mode 1, light gray: mode 2, medium gray: keep current mode unchanged, white: uncontrollable states); Switching signal generated by the lazy controller and corresponding evolution of the state variables, switching signal has dwell time $2$ (center); Associated trajectory of the switched system (right).}
\label{fig:ex4}
\end{center}
\end{figure}

\section{Conclusion}

In this paper, we showed, under assumptions ensuring incremental stability, such as existence of a common $\delta$-GAS Lyapunov function or multiple $\delta$-GAS Lyapunov functions with dwell time, the existence of approximately bisimilar symbolic abstractions for switched systems. The proof of existence is constructive: these
abstractions are effectively computable and any precision can be achieved. Two non-trivial examples of controller design based on symbolic models of switched systems have been shown.

The authors are currently improving the presented results in two different directions. The controllers resulting from arbitrary specifications may require switching surfaces with complex geometries. This increases the space complexity of controllers and complicates its real-time implementation. To address this difficulty, the authors are currently investigating the synthesis of more conservative controllers that are guaranteed to have lower complexity switching regions. The other direction being investigated is the most efficient enforcement of the dwell time requirement. Instead of building this requirement in the symbolic model, which results in larger symbolic models, it is possible to incorporate this requirement as part of the overall specification. We can thus synthesize controllers based on smaller symbolic models while meeting all the dwell time requirements.


\bibliographystyle{alpha}
\bibliography{ref}

\begin{thebibliography}{AVW03}

\bibitem[Ang02]{angeli2002}
D.~Angeli.
\newblock A {Lyapunov} approach to incremental stability properties.
\newblock {\em IEEE Trans. on Automatic Control}, 47(3):410--421, March 2002.

\bibitem[AS99]{angeli1999}
D.~Angeli and E.D. Sontag.
\newblock Forward completeness, unboundedness observability, and their
  {Lyapunov} characterizations.
\newblock {\em Systems and Control Letters}, 38(3):209--217, 1999.

\bibitem[AVW03]{arnold2003}
A.~Arnold, A.~Vincent, and I.~Walukiewicz.
\newblock Games for synthesis of controllers with partial observation.
\newblock {\em Theoretical Computer Science}, 28(1):7--34, 2003.

\bibitem[BPM05]{beccuti2005}
A.G. Beccuti, G.~Papafotiou, and M.~Morari.
\newblock Optimal control of the boost dc-dc converter.
\newblock In {\em IEEE Conf. on Decision and Control}, pages 4457--4462, 2005.

\bibitem[BPM06]{beccuti2006}
A.G. Beccuti, G.~Papafotiou, and M.~Morari.
\newblock Explicit model predictive control of the boost dc-dc converter.
\newblock In {\em Analysis and Design of Hybrid Systems}, pages 315--320, 2006.

\bibitem[BRC05]{buisson2005}
J.~Buisson, P.Y. Richard, and H.~Cormerais.
\newblock On the stabilisation of switching electrical power converters.
\newblock In {\em Hybrid Systems: Computation and Control}, volume 3414 of {\em
  LNCS}, pages 184--197. Springer, 2005.

\bibitem[GP07]{girard2007}
A.~Girard and G.J. Pappas.
\newblock Approximation metrics for discrete and continuous systems.
\newblock {\em IEEE Trans. on Automatic Control}, 52(5):782--798, 2007.

\bibitem[GPT08]{girard2008}
A.~Girard, G.~Pola, and P.~Tabuada.
\newblock Approximately bisimilar symbolic models for incrementally stable
  switched systems.
\newblock In {\em Hybrid Systems: Computation and Control}, volume 4981 of {\em
  LNCS}, pages 201--214. Springer, 2008.

\bibitem[Kha96]{khalil1996}
H.K. Khalil.
\newblock {\em Nonlinear Systems}.
\newblock Prentice Hall, 1996.

\bibitem[Lib03]{liberzon2003}
D.~Liberzon.
\newblock {\em Switching in Systems and Control}.
\newblock Birkhauser, 2003.

\bibitem[Mil89]{milner1989}
R.~Milner.
\newblock {\em Communication and Concurrency}.
\newblock Prentice Hall, 1989.

\bibitem[Par81]{park1981}
D.M.R. Park.
\newblock Concurrency and automata on infinite sequences.
\newblock In {\em GI-Conf. on Theoretical Computer Science}, volume 104 of {\em
  LNCS}, pages 167--183. Springer, 1981.

\bibitem[PGT07]{pola2007}
G.~Pola, A.~Girard, and P.~Tabuada.
\newblock Approximately bisimilar symbolic models for nonlinear control
  systems.
\newblock In {\em IEEE Conf. on Decision and Control}, 2007.

\bibitem[PW96]{praly1996}
L.~Praly and Y.~Wang.
\newblock Stabilization in spite of matched unmodeled dynamics and an
  equivalent definition of input-to-state stability.
\newblock {\em Math. Control Signals Systems}, 9:1--33, 1996.

\bibitem[RW87]{ramadge1987}
P.J. Ramadge and W.M. Wonham.
\newblock Supervisory control of a class of discrete event systems.
\newblock {\em SIAM Journal on Control and Optimization}, 25(1):206--230, 1987.

\bibitem[SEK03]{senesky2003}
M.~Senesky, G.~Eirea, and T.J. Koo.
\newblock Hybrid modelling and control of power electronics.
\newblock In {\em Hybrid Systems: Computation and Control}, volume 2623 of {\em
  LNCS}, pages 450--465. Springer, 2003.

\bibitem[Tab06]{SymbolicControl06}
P.~Tabuada.
\newblock Symbolic control of linear systems based on symbolic subsystems.
\newblock {\em IEEE Trans. on Automatic Control}, 51(6):1003--1013, June 2006.

\bibitem[Tab08]{tabuada2007}
P.~Tabuada.
\newblock An approximate simulation approach to symbolic control.
\newblock {\em IEEE Trans. on Automatic Control}, 2008.
\newblock To appear.

\end{thebibliography}

\end{document}